\documentclass{article}
\usepackage{hyperref}
\usepackage{graphicx}
\usepackage[all]{xy} 
\usepackage{mathrsfs} 
\usepackage{amsmath}
\usepackage{amsthm}
\usepackage{yhmath}
\usepackage{array} 
\usepackage{verbatim} 
\usepackage{amssymb}
\usepackage{enumerate}
\usepackage{calc}
\usepackage{colortbl}
\usepackage{txfonts}
\usepackage{textcomp}
\usepackage{tabularx} 
\usepackage[table]{xcolor}
\usepackage{hhline}
\usepackage{lscape}
\usepackage{rotating}
\usepackage{tensor}
\usepackage{arydshln} 
\usepackage{bbm} 
\usepackage{appendix} 
\usepackage{placeins} 
\usepackage[T1]{fontenc} 
\usepackage{titletoc}
\usepackage{pgf}
\usepackage{tikz, ifthen}
\usepackage[nomessages]{fp}
\usetikzlibrary{matrix,arrows}
\usetikzlibrary{patterns}
\usetikzlibrary{shapes,decorations,shadows}
\usepackage{ytableau}
\usepackage[enableskew]{youngtab}

\DeclareMathOperator{\End}{End} 
\DeclareMathOperator{\Ext}{Ext}

\DeclareMathOperator{\Hom}{Hom} 
\DeclareMathOperator{\rad}{rad}

\DeclareMathOperator{\rk}{rk} 
\DeclareMathOperator{\sh}{sh}

\DeclareMathOperator{\rep}{rep}

\DeclareMathOperator{\GL}{GL}

\DeclareMathOperator{\Ad}{Ad}

\definecolor{lightgreen}{rgb}{0.8,1,0.8}
\definecolor{dlightgreen}{rgb}{0.6,1,0.6}
\definecolor{lightblue}{rgb}{0.64,0.83,0.93}
\definecolor{dlightblue}{rgb}{0.39,0.58,0.93}
\definecolor{lightviolet}{rgb}{0.85,0.44,0.84}
\definecolor{dlightviolet}{rgb}{0.7,0.23,0.94}
\definecolor{grey}{rgb}{0.8,0.8,0.8}
\definecolor{lblue}{rgb}{0.3,0.0,4.4}
\definecolor{lred}{rgb}{4.3,0.0,0.4}
 \definecolor{lightblue}{rgb}{0.8,0.8,0.9}
  \definecolor{lightred}{rgb}{0.9,0.8,0.8}

\newcommand*{\punkte}{\dots\unkern}
\newcolumntype{C}[1]{>{\centering\arraybackslash}p{#1}}
\newcommand{\A}{\mathcal{A}} 
\newcommand{\B}{\mathcal{B}} 
 
\newcommand{\Pa}{\mathcal{P}}

\newcommand{\Q}{\mathcal{Q}} 
\newcommand{\Orb}{\mathcal{O}} 
\newcommand{\N}{\mathcal{N}}

\newcommand{\dimv}{\underline{\dim}}
 
\newcommand{\df}{\underline{d}}

\newcommand{\CC}{\mathcal{C}}

\newtheorem{theorem}{Theorem}[section]
\newtheorem{lemma}[theorem]{Lemma}
\newtheorem{definition}[theorem]{Definition}
\newtheorem{proposition}[theorem]{Proposition}
\newtheorem{corollary}[theorem]{Corollary}

\newtheorem{example}[theorem]{Example}

\newtheorem*{1maintheorem}{Main Theorem}
\newtheorem*{classtheorem}{Classification Theorem}
\newtheorem*{theorem*}{Theorem}

\newcounter{x}
\newcounter{y}
\newcounter{z}
\newcommand*\cubecolors[1]{%
  \ifcase#1\relax
  \or\colorlet{cubecolor}{cyan}%
  \or\colorlet{cubecolor}{blue}%
  \or\colorlet{cubecolor}{purple}%
  \or\colorlet{cubecolor}{red}%
  \or\colorlet{cubecolor}{purple}%
  \or\colorlet{cubecolor}{blue}%
  \else
    \colorlet{cubecolor}{white}%
  \fi
}

\newcommand\yaxis{180}
\newcommand\zaxis{-27}
\newcommand\xaxis{90}

\newcommand\topside[3]{
  \fill[fill=cubecolor, draw=black,shift={(\xaxis:#1)},shift={(\yaxis:#2)},
  shift={(\zaxis:#3)}] (0,0) -- (1,0) -- (0.5,0.25) --(-0.5,0.25)--(0,0);
}

\newcommand\leftside[3]{
  \fill[fill=cubecolor, draw=black,shift={(\xaxis:#1)},shift={(\yaxis:#2)},
  shift={(\zaxis:#3)}] (0,0) -- (0,-1) -- (-0.5,-0.75) --(-0.5,0.25)--(0,0);
}

\newcommand\rightside[3]{
  \fill[fill=cubecolor, draw=black,shift={(\xaxis:#1)},shift={(\yaxis:#2)},
  shift={(\zaxis:#3)}] (0,0) -- (1,0) -- (1,-1) --(0,-1)--(0,0);
}

\newcommand\cube[3]{
  \topside{#1}{#2}{#3} \leftside{#1}{#2}{#3} \rightside{#1}{#2}{#3}
}

\newcommand\planepartition[2]{
 \setcounter{x}{0}
 \foreach \a in {#2} {
    \addtocounter{x}{1}
    \setcounter{y}{-1}
    \cubecolors{\value{x}}
    \foreach \b in \a {
      \addtocounter{y}{1}
      \setcounter{z}{-1}
      \foreach \c in {0,...,\b} {
        \addtocounter{z}{1}
      \ifthenelse{\c=0}{\setcounter{z}{-1},\addtocounter{y}{0}}{
        \FPeval{\newz}{clip(0.55*\the\value{z})}
        \cube{\value{x}}{\value{y}}{\newz};
        \FPeval{\result}{clip(#1-\the\value{y}+2*\the\value{z})}
        \draw[draw=black,shift={(\xaxis:\value{x})},shift={(\yaxis:\value{y})},
  shift={(\zaxis:\newz)}] (0.5,-0.5) node {};}
      }
    }
  }
}

\newcommand\planepartitionD[2]{
 \setcounter{x}{0}
 \foreach \a in {#2} {
    \addtocounter{x}{1}
    \setcounter{y}{-1}
    \cubecolors{\value{x}}
    \foreach \b in \a {
      \addtocounter{y}{1}
      \setcounter{z}{-1}
      \foreach \c in {0,...,\b} {
        \addtocounter{z}{1}
      \ifthenelse{\c=0}{\setcounter{z}{-1},\addtocounter{y}{0}}{
        \FPeval{\newz}{clip(0.55*\the\value{z})}
        \cube{\value{x}}{\value{y}}{\newz};
        \FPeval{\result}{clip(#1-\the\value{y}-2*\the\value{z})}
        \draw[draw=black,shift={(\xaxis:\value{x})},shift={(\yaxis:\value{y})},
  shift={(\zaxis:\newz)}] (0.5,-0.5) node {};}
      }
    }
  }
}

\newcommand\planepartitionE[2]{
 \setcounter{x}{0}
 \foreach \a in {#2} {
    \addtocounter{x}{1}
    \setcounter{y}{-1}
    \cubecolors{\value{x}}
    \foreach \b in \a {
      \addtocounter{y}{1}
      \setcounter{z}{-1}
      \foreach \c in {0,...,\b} {
        \addtocounter{z}{1}
      \ifthenelse{\c=0}{\setcounter{z}{-1},\addtocounter{y}{0}}{
        \FPeval{\newz}{clip(0.55*\the\value{z})}
        \cube{\value{x}}{\value{y}}{\newz};
        \FPeval{\result}{clip(#1-\the\value{y}-2*\the\value{z})}
        \draw[draw=black,shift={(\xaxis:\value{x})},shift={(\yaxis:\value{y})},
  shift={(\zaxis:\newz)}] (0.5,-0.5) node {\result};}
      }
    }
  }
}

\makeatletter
\long\def\nnfoottext#1{\insert\footins{\footnotesize
    \interlinepenalty\interfootnotelinepenalty
    \splittopskip\footnotesep
    \splitmaxdepth \dp\strutbox \floatingpenalty \@MM
    \hsize\columnwidth \@parboxrestore
   \edef\@thefnmark{}
   \edef\@currentlabel{}\@makefntext
    {\rule{\z@}{\footnotesep}\ignorespaces
      #1\strut}}}
\makeatother
\begin{document}
\parindent0pt
\title{\bf Multi-graded nilpotent tuples}
\nnfoottext{Keywords: quiver with relation, graded nilpotent pairs, multi-grading, representation type, tame, wild, finiteness criterion

AMS Classification 2010: 16G20, 16G60, 17B08}

\author{Magdalena Boos\\Ruhr University Bochum\\ Faculty of Mathematics\\  D - 44780 Bochum, Germany.\\  Magdalena.Boos-math@ruhr-uni-bochum.de }
\date{}
\maketitle

\begin{abstract}
We discuss multi-graded nilpotent tuples of multi-graded vector spaces which are a generalization of graded nilpotent pairs \cite{B4}. The multi-grading yields a natural notion of a shape of such tuple and our main interest is to answer the question \textquotedblleft Is the number of multi-graded nilpotent tuples of a fixed shape, up to base change in the homogeneous components, finite?\textquotedblright~ Our methods make use of a translation to the class of  so-called \textquotedblleft Multi-staircase algebras\textquotedblright~ and we classify their representation types. 
\end{abstract}

\section{Introduction}\label{sect:intro}
Given a complex semi-simple Lie algebra $\mathfrak{g}$, an element $x\in\mathfrak{g}$ is called \textit{regular} if its centralizer  in $\mathfrak{g}$
has dimension $\rk \mathfrak{g}$. A regular nilpotent element is called \textit{principal nilpotent} -- the set of these elements is studied in detail in \cite{Ko1,Ko2}.

Let $G$ be the Lie group corresponding to $\mathfrak{g}$. Ginzburg appplies the aforementioned notions in \cite{Gi} to the commuting variety of $\mathfrak{g}$ by defining a \textit{principal nilpotent pair} to be a  pair $x = (x_1, x_2)\in\mathfrak{g}^2$ of commuting regular elements with simultaneous centralizer of dimension $\rk \mathfrak{g}$, such that for all $(t_1,t_2)\in\mathbf{C}^*\times\mathbf{C}^*$, there is  an element $g=g(t_1,t_2)\in G$ with $(t_1 x_1,t_2 x_2)=(g\cdot x_1,g\cdot x_2)$.  In particular, $x_1$ and $x_2$ are nilpotent elements of $\mathfrak{g}$.

The $\Ad G$--diagonal action on the set of principal nilpotent pairs admits only a finite number of orbits; and a classification of these for $\mathfrak{g}=\mathfrak{sl}_n(\mathbf{C})$ is given by the set of Young diagrams with $n$ boxes \cite{Gi}. In particular, the Young diagram determines one pair $(x_1,x_2)$ in the corresponding orbit explicitly: the boxes of the Young diagram correspond in a particular way to the standard basis vectors of $\mathbf{C}^n$, and the element $x_1$ moves these from left to right whereas $x_2$ moves them from top to bottom.  Elashvili and Panyushev classify the orbits for the remaining classical types in \cite{EP1} and for the exceptional types in \cite{EP2}. He shows that the number of $G$-orbits of these pairs is finite and classifies them for type A.  Further representation-theoretic results are found in \cite{Gi2} where he relates principal nilpotent pairs with subschemes of the normalization of the isospectral commuting variety. 

This setup is generalized as follows for $\mathfrak{g}=\mathfrak{gl}_n(\mathbf{C})$ in \cite{B4}:
Every graded vector space $V=\bigoplus_{(x,y)\in\mathbf{Z}_{\geq 1}^2} V_{x,y}$  determines a Young diagram as defined in Subsection \ref{ssect:gener_YD} which is called the \textit{shape} of $V$. A  graded nilpotent pair of $V$ is then defined to be a pair $(\varphi, \psi)$ of commuting nilpotent operators of $V$ , such that both nilpotent operators are compatible with the bi-grading in a natural commuting way: $\varphi$ respects the bigrading in horizontal direction $\varphi(V_{x,y})\subseteq V_{x-1,y}$ and $\psi$ respects the bigrading vertically $\psi(V_{x,y})\subseteq V_{x,y-1}$. In particular, a corresponding principal nilpotent pair appears in the situation, where $\dim V_{x,y}=1$ for every $x,y$. As we deal with graded vectorspaces, we only consider these tuples up to the action given by base change in the homogeneous components; in \cite{B4} a criterion is proved which states for each fixed shape, if there is only a finite number of orbits for every graded vector space of this shape.

In this article, we consider a further generalization of this setup, namely multi-graded, or $k$-graded, nilpotent tuples of a $k$-graded vector space $V= \bigoplus_{(x_1,..,x_k)\in\mathbf{Z}_{\geq 1}^k} V_{x_1,..,x_k}$ for an integer $k$. These are tuples  $(\varphi_1,...,\varphi_k)$ of pairwise commuting nilpotent operators on $V$ which respect the  grading in every direction $i$:  \[\varphi_i\mid_{V_{x_1,..,x_k}}(V_{x_1,..,x_k})\subseteq V_{x_1,...,x_i-1,...,x_k}.\]

The non-zero components of the multi-grading of $V$ induce a $k$-dimensional (generalized)  Young diagram $\Lambda$, that is, a multi-dimensional Young diagram in a natural way and call every corresponding $k$-graded nilpotent tuple $\Lambda$-shaped, thus. 

Our main result gives a complete and explicit answer to a standard Lie-theoretic question: Are there only finitely many $\Lambda$-shaped multi-graded nilpotent tuples up to base change by a Levi respecting the grading? 

The methods of proof are mainly taken from Representation Theory of finite-dimensional algebras. For each generalized Young diagram $\Lambda$, we define a so-called multi-staircase algebra $\A(\Lambda)$ in Section \ref{ssect:translation} which generalizes the concept of staircase algebras as in \cite{B4}: we define a quiver $\Q(\Lambda)$ whose vertices are in bijection to the boxes of $\Lambda$; and there is  an arrow from vertex $i$ to vertex $j$ if the box corresponding to $i$ is an above neighbour of the box corresponding to $j$ in one of the $k$ directions of $\Lambda$. Then $\A(\Lambda)$ is defined to be the path algebra of $\Q(\Lambda)$ with (roughly speaking) all commutativity relations which are determined by the squares of $\Lambda$. The above posed question is then answered by  discussing the representation types of all multi-staircase algebras which leads to our Classification Theorem.
\begin{classtheorem}\label{thm:class}
Let $\Lambda$ be a generalized proper $k$-dimensional Young diagram. The multi-staircase algebra $\A(\Lambda)$ is  
\begin{enumerate}
\item of finite representation type if and only if
$\Lambda$ is listed as a (generalized) Young diagram in Theorem \ref{thm:stair_rep_type} (a) or Theorem  \ref{thm:tri_stair_rep_type} (a),

\item tame concealed if and only if
$\Lambda$ is listed as a Young diagram in Theorem \ref{thm:stair_rep_type} (b), Theorem  \ref{thm:tri_stair_rep_type} (b) or Theorem \ref{thm:high_stair_rep_type} (b).

\item tame non-concealed if and only if $\Lambda$ is listed as a Young diagram in Theorem \ref{thm:stair_rep_type} (c), Theorem  \ref{thm:tri_stair_rep_type} (c) or Theorem \ref{thm:high_stair_rep_type} (c).
\end{enumerate}
Otherwise, and in particular if $k\geq 5$, the algebra $\A(\Lambda)$ is of wild representation type. 
\end{classtheorem}
Note that staircase algebras are discussed in \cite{B4} and the special case of rectangular partitions can be found in \cite{Le}, where $3$-dimensional results are available, too.

By Theorem \ref{thm:crit_multi_grad_pairs}, this classification provides our Main Theorem about $\Lambda$-shaped multi-graded nilpotent tuples as a direct consequence.
\begin{1maintheorem}
Let $\Lambda$ be a generalized $k$-dimensional partition. The number of $\Lambda$-shaped multi-graded nilpotent tuples (up to Levi-base change) is finite if and only if 
the algebra $\A(\Lambda)$ is of finite representation type. These cases are those of Classification Theorem (a). In particular, whenever $k\geq 4$, there are infinitely many such pairs (see Corollary  \ref{cor:high_grad_inf_tuples}).
\end{1maintheorem}
 Results are extended by considering certain minimal nullroots of particular tame concealed multi-staircase algebras which yield information about multi-graded nilpotent tuples of fixed vector spaces $V$.\\[2ex]
{\bf Acknowledgments:} The author would like to thank Markus Reineke for interesting discussions on the subject and Edward Richmond for his support with drawing generalized Young diagrams in Latex.

\section{Theoretical background}\label{sect:theo_backgrounnd}
We sketch some theoretical background now which we will refer to later when studying our main questions.
Let $K=\overline{K}$ be an algebraically closed field and let $\GL_n\coloneqq\GL_n(K)$ be the general linear group for a fixed integer $n\in\textbf{N}$ regarded as an affine variety. 
\subsection{Representation Theory (of finite-dimensional algebras)}\label{ssect:theory}
The concepts of this subsection can mostly be found in \cite{ASS, Pi}. A \textit{finite quiver} $\Q$ is a directed graph $\Q=(\Q_0,\Q_1,s,t)$, given by a finite set of \textit{vertices} $\Q_0$, a finite set of \textit{arrows} $\Q_1$ and two maps  $s,t:\Q_1\rightarrow \Q_0$ defining the source $s(\alpha)$ and the target $t(\alpha)$ of an arrow $\alpha$. 
A \textit{path} is a sequence of arrows $\omega=\alpha_s\punkte\alpha_1$, such that $t(\alpha_{k})=s(\alpha_{k+1})$ for all $k\in\{1,\punkte,s-1\}$.

 We define the \textit{path algebra} $K\Q$ as the $K$-vector space with a basis consisting of all paths in $\Q$, formally included is a path $\varepsilon_i$ of length zero for each $i\in \Q_0$ starting and ending in $i$. The multiplication  is defined by concatenation of paths, if possible, and equals $0$, otherwise.
 
Let us define the \textit{arrow ideal} $R_{\Q}$ of $K\Q$ as the (two-sided) ideal which is generated (as an ideal) by all arrows in $\Q$. In particular, the arrow ideal of $\Q$ equals the radical of $K\Q$ if $\Q$ does not contain oriented cycles. An arbitrary ideal $I\subseteq K\Q$ is called \textit{admissible} if there exists an integer $s$, such that $R_{\Q}^s\subseteq I\subseteq R_{\Q}^2$. Given such admissible ideal $I$, we denote by $I(i,j)$ the set of paths in $I$ starting in $i$ and ending in $j$.

Given such admissible ideal $I$, the path algebra of $\Q$, \textit{bound by} $I$ is defined as $\A:=K\Q/I$; it is a basic and finite-dimensional $K$-algebra \cite{ASS}; the elements of $I$ are the so-called \textit{relations} of $\A$. Then $\A$ is called \textit{triangular}, if $\Q$ does not contain an oriented cycle. Let us fix such triangular algebra $\A=K\Q/I$ in the following.

Let $\rep\A$ be the abelian $K$-linear category of all finite-dimensional $\A$-representations which is equivalent to the category of  \textit{$K$-representations} of $\Q$, which are \textit{bound by $I$}, defined as follows: 
The objects are given by tuples 
$((M_i)_{i\in \Q_0},(M_\alpha)_{\alpha\in \Q_1})$, 
where the $M_i$ are $K$-vector spaces, and the $M_\alpha\colon M_{s(\alpha)}\rightarrow M_{t(\alpha)}$ are $K$-linear maps. For each path $\omega$ in $\Q$ as above, we denote $M_\omega=M_{\alpha_s}\cdot\punkte\cdot M_{\alpha_1}$ and ask a representation $M$ to fulfill $\sum_\omega\lambda_\omega M_\omega=0$ whenever $\sum_\omega\lambda_\omega\omega\in I$. A \textit{morphism of representations} $M=((M_i)_{i\in \Q_0},(M_\alpha)_{\alpha\in \Q_1})$ and
 \mbox{$M'=((M'_i)_{i\in \Q_0},(M'_\alpha)_{\alpha\in \Q_1})$} consists of a tuple of $K$-linear maps $(f_i\colon M_i\rightarrow M'_i)_{i\in \Q_0}$, such that $f_jM_\alpha=M'_\alpha f_i$ for every arrow $\alpha\colon i\rightarrow j$ in $\Q_1$. 
 
The \textit{dimension vector} of an $\A$-representation $M$  is defined by $\dimv M\in\mathbf{N}^{\Q_0}$; in more detail $(\dimv M)_{i}=\dim_K M_i$ for $i\in \Q_0$. 
Let us fix such a dimension vector $\df\in\mathbf{N}^{\Q_0}$ and denote by $\rep\A(\df)$ the full subcategory of $\rep\A$ which consists of all representations of dimension vector $\df$.

By defining the affine space $R_{\df}K\Q:= \bigoplus_{\alpha\colon i\rightarrow j}\Hom_K(K^{d_i},K^{d_j})$, it is clear that its points $m$ naturally correspond to representations $M\in\rep K\Q(\df)$ with $M_i=K^{d_i}$ for $i\in \Q_0$. 
 Via this correspondence, the set of such representations bound by $I$ corresponds to a closed subvariety $R_{\df}\A\subset R_{\df}K\Q$.
 
The algebraic group $\GL_{\df}=\prod_{i\in \Q_0}\GL_{d_i}$ acts on $R_{\df}K\Q$ and on $R_{\df}\A$ via base change, furthermore the $\GL_{\df}$-orbits $\Orb_M$ of this action are in bijection to the isomorphism classes of representations $M$ in $\rep\A(\df)$.

A non-zero representation $M$ is called \textit{indecomposable}, if it cannot be decomposed into a direct sum $M\cong M_1\oplus M_2$, where $M_1$ and $M_2$ are non-zero representations. Due to Krull, Remak and Schmidt, every finite-dimensional $\A$-representation decomposes into an -- up to isomorphism -- unique direct sum of indecomposables. For certain classes of finite-dimensional algebras, a convenient tool for the classification of these indecomposable representations is the \textit{Auslander-Reiten quiver} $\Gamma_{\A}=\Gamma(\Q,I)$ of $\rep K\Q/I$. Its vertices $[M]$ are given by the isomorphism classes of indecomposable representations of $\rep K\Q/I$; the arrows between two such vertices $[M]$ and $[M']$ are parametrized by a basis of the so-called \textit{space of irreducible morphisms}. One standard technique to calculate $\Gamma_{\A}$ is the \textit{knitting process} (see, for example, \cite[IV.4]{ASS}).

Let $\A= K\Q/I$ be an algebra of a quiver with only commutativity and zero relations. We say that an algebra $\B = K\Q'/I'$ is a \textit{convex subcategory} of $\A$, if $\Q'$ is a pathclosed \textit{subquiver} of $\Q$ (that is, for every two vertices $x$ and $y$ in $\Q'$, every path of $\Q$ from $x$ to $y$ is a path in $\Q'$, as well) and $I'\coloneqq \langle I(i,j)\mid i,j\in \Q'_0\rangle$. 

An indecomposable projective $P_i$ at vertex $i$ has a so-called \textit{separated} radical, if for arbitrary two non-isomorphic direct summands of its radical, their supports are contained in different components of the subquiver $\Q$ obtained by deleting all starting points of paths ending in $i$. We say that $\A$  \textit{fulfills the separation condition}, if every projective indecomposable has a separated radical. If this condition is fulfilled, $\A$ admits a preprojective component, see \cite{Bo4}. Recall that a preprojective component is a component  of the Auslander-Reiten quiver of  $\A$ which consists only of Auslander-Reiten translates \cite{ASS} of projective representations and does not contain oriented cycles \cite{HaRi}.  

In general, the definition of an algebra to be \textit{strongly simply connected} algebra is quite involved. In case of a triangular algebra $\A$, there is an equivalent description, though : $\A$ is \textit{strongly simply connected} if and only if every convex subcategory of $\A$ satisfies the separation condition \cite{Sko2}.

The algebra $\A$ is called of \textit{finite representation type}, if there are only finitely many isomorphism classes of indecomposable representations. If it is not of finite representation type, the algebra is of \textit{infinite representation type}. These infinite types split up into exactly two disjoint cases by Drozd's Dichotomy statement \cite{Dr}: 
 \begin{itemize}
 \item  $\A$ has \textit{tame representation type} (or \textit{is tame}) if for every integer $n$ there is an integer $m_n$ and there are finitely generated $K[x]$-$\A$-bimodules $M_1,\punkte,M_{m_n}$ which are free over $K[x]$, such that for all but finitely many isomorphism classes of indecomposable right $\A$-modules $M$ of dimension $n$, there are elements $i\in\{1,\punkte,m_n\}$ and $\lambda\in K$, such that  $M\cong  K[x]/(x-\lambda)\otimes_{K[x]}M_i$.
 \item Or $\A$ has \textit{wild representation type} (or \textit{is wild}) if there is a finitely generated $K\langle X,Y\rangle$-$\A$-bimodule $M$ which is free over $K\langle X,Y\rangle$ and sends non-isomorphic finite-dimensional indecomposable $K\langle X,Y\rangle$-modules via the functor $\_\otimes_{K\langle X,Y\rangle}M$ to non-isomorphic indecomposable $\A$-modules.
\end{itemize} 
The notion of a tame algebra $\A$ yields that there are at most one-parameter families of pairwise non-isomorphic indecomposable $\A$-modules; in the wild case there are parameter families of arbitrary many parameters.

For a triangular algebra $\A = K\Q/I$, the \textit{Tits form} $q_{\A}:\mathbf{Z}^{\Q_0}\rightarrow \mathbf{Z}$  is the integral quadratic form defined by 
\[q_{\A}(v) = \sum_{i\in\Q_0} v_i^2 - \sum_{\alpha:i\rightarrow j\in\Q_1} v_iv_j + \sum_{i,j\in\Q_0} r(i,j)v_iv_j;\]
here $r(i,j)$ equals the number of elements in $R\cap I(i,j)$ whenever $R$ is a minimal set of generators of $I$, such that $R\subseteq \bigcup_{i,j\in\Q_0} I(i,j)$.
The quadratic form $q_{\A}$ is called \textit{weakly positive}, if $q_{\A}(v) > 0$ for every $v\in\mathbf{N}^{\Q_0}$; and \textit{(weakly) non-negative}, if $q_{\A}(v) \geq 0$ for every $v\in\mathbf{Z}^{\Q_0}$ (or $v\in\mathbf{N}^{\Q_0}$, respectively). If $q_{\A}$ is weakly non-negative, we define its \textit{radical} by $\rad q_{\A}:=\{u\in\mathbf{Z}^{\Q_0} \mid q_{\A}(u)=0\}$ and call its elements \textit{nullroots}.

The definiteness of the Tits form is closely related to the representation type of $\A$, and many results are, for example, summarized by De la Pe\~na and Skowro\'{n}ski in \cite{DlPS}.
The following criterion for finite representation type is due to Bongartz \cite{Bo4}.
\begin{theorem}\label{thm:crit_finite}
Let $\A = K\Q/I$ be a triangular algebra, which admits a preprojective component. Then $\A$ is representation-finite if and only if the Tits form $q_{\A}$ is weakly positive. Furthermore, if the equivalent conditions hold true, then the dimension vector function $X\mapsto \dim X$ induces a bijection between the set of isomorphism classes of indecomposable $\A$-modules and the set of positive roots of $q_{\A}$.
\end{theorem}
Assume that $\Gamma_{\A}$ has a preprojective component. The algebra $\A$ is called  \textit{critical} if $q_{\A}$ is not weakly positive, but every proper restriction of $q_{\A}$ is weakly positive. The term "critical" is actually intuitive, since the conditions of Theorem \ref{thm:crit_finite} are equivalent to $\A$ not having a convex subcategories which is critical. The critical algebras are the tame concealed algebras which are listed in the famous Bongartz-Happel-Vossieck list (BHV-list)  \cite{Bo5,HaVo,GaRo}. These are the minimal tame ones \cite{HaVo}, thus: if an algebra contains a proper convex subcategory which is tame concealed, it cannot be tame concealed itself.

 The algebra $\A$ is called \textit{hypercritical} if $q_{\A}$ is not weakly non-negative while every proper restriction of $q_{\A}$ is weakly non-negative. The hypercritical algebras have been classified \cite{Un}. For strongly simply connected algebras, they turn out to be the minimal wild algebras as the following classification of tame types due to of Br\"ustle, De la Pe\~na and Skowro{\'n}ski yields \cite{BdlPS}.
 \begin{theorem}\label{thm:crit_tame}
 Let $\A$ be strongly simply connnected of infinite representation type. Then  the following are equivalent:
\begin{enumerate}
\item $\A$ is tame;
\item $q_{\A}$ is weakly non-negative;
\item $\A$ does not contain a full convex subcategory which is hypercritical.
\end{enumerate}
\end{theorem} 
By dichotomy, Theorem \ref{thm:crit_tame} yields a sufficient criterion for wildness.
\begin{corollary}\label{cor:crit_wild}
 Let $\A$ be strongly simply connnected. Whenever there exists $v\in \mathbf{N}^{\Q_0}$, such that $q_{\A}(v)\leq -1$, then $\A$ is of wild representation type.
\end{corollary}
\subsection{Generalized Young diagrams}\label{ssect:gener_YD}
Let $N\in\mathbf{N}$ be an integer, then a \textit{partition} of $N$ is a decreasing tuple $(\lambda_1,...,\lambda_k)$ of positive integers, such that $\lambda_1+...+\lambda_k=N$ for some $k\in\mathbf{N}$; let $\Pa(N)$ denote the set of partitions of $N$ and by $\Pa$ denote the set of all partitions of integers.  For easier notation, we merge similar entries of partitions by potencies, for example $(1,1,2,2,2,6,8,8)=: (1^2,2^3,6,8^2)$.  Corresponding to $\lambda$ there is a \textit{Young diagram}; we fill its boxes with tuples $(x,y)$ as sketched in the example and identify partitions and Young diagrams:

\begin{example}\label{ex:partition}
Let $N=10$ and $\lambda=(3,3,2,1,1)$. Then the Young diagram is
 \begin{center}
 \begin{ytableau}
\none &\none&\none&_{(2,3)}& _{(1,3)} \\
\none &\none&_{(3,2)}& _{(2,2)} & _{(1,2)}  \\
_{(5,1)}&_{(4,1)} & _{(3,1)}&_{(2,1)} &_{(1,1)}\\
\end{ytableau}
\end{center}
\end{example}
Thus, we can identify each partition, or each Young diagram, with a subset of $\mathbf{N}_{\geq 1}\times \mathbf{N}_{\geq 1}$. Furthermore, a subset $X\subseteq\mathbf{N}_{\geq 1}\times \mathbf{N}_{\geq 1}$ defines a partition if and only if the condition "If $(x,y)\in X$, then $(a,b)\in X$ whenever $a\leq x$ and $b\leq y$" is fulfilled. This in particular yields
\[YD^{(2)}\coloneqq \{X\subseteq\mathbf{N}_{\geq 1}\times \mathbf{N}_{\geq 1}\mid (x,y)\in X \Rightarrow (a,b)\in X,~ \forall a\leq x, b\leq y\} \cong\Pa.\]

This setup can be generalized to $k\in \mathbf{N}$ dimensions as follows: Let
\[YD^{(k)}\coloneqq \{X\subseteq\mathbf{N}_{\geq 1}^k\mid (x_1,...,x_k)\in X \Rightarrow (a_1,...,a_k)\in X,~ \forall a_i\leq x_i, \forall i\}\]
We call an element of $YD^{(k)}$ a \textit{$k$-dimensional} or \textit{generalized Young diagram}.

Another special case of generalized Young diagrams is of particular interest: namely, $3$-dimensional Young diagrams which in literature are usually called \textit{plane partitions}. These are  the subject of study of several articles and books, either as combinatorial objects themselves or in the context of  deeper theory, see for example \cite{And,Mac}. We subdivide them into two kinds, flat ones and non-flat ones.

Let $\Lambda$ be a $3$-dimensional Young diagram, then $\Lambda$ is called \textit{flat}, if one of the following conditions is fulfilled:
\begin{itemize}
\item $\Lambda\subseteq \{(1,x,y),(x,1,y),\mid x,y\geq 1\}$
\item $\Lambda\subseteq \{(1,x,y),(x,y,1)\mid x,y\geq 1\}$
\item $\Lambda\subseteq \{(x,1,y),(x,y,1)\mid x,y\geq 1\}$
\end{itemize}

In particular, every flat $3$-dimensional Young diagram is given by two decreasing partitions $\lambda=(\lambda_1,...,\lambda_k)$ and $\mu=(\mu_1,...,\mu_l)$, such that $\lambda_1=\mu_1=:x$. We denote $\Lambda=(\lambda_k,...,\lambda_2,\underline{x},\mu_2,...,\mu_l)$. If $\Lambda$ is not flat, it is called \textit{non-flat}.

Graphically speaking, every flat $3$-dimensional Young Diagram can be \textquotedblleft properly embedded into the plane\textquotedblright. In more detail, if for example 
\[\Lambda\subseteq \{(1,x,y),(x,1,y),\mid x,y\geq 1\},\]
 then consider the subset of $\Lambda$ of tuples where the first entry equals $1$. By deleting this entry in every tuple, we obtain a partition $\lambda=(\lambda_1,...,\lambda_k)$. In the same way, deletion of the second entry $1$ gives us a partition $\mu=(\mu_1,...,\mu_l)$. Since 
 \[\lambda_1=\mu_1=\sharp\{(1,1,y)\mid y\geq 1, (1,1,y)\in\Lambda\},\] the embedded diagram is $\Lambda=(\lambda_k,...,\lambda_2,\underline{\lambda_1},\mu_2,...,\mu_l)$.
\begin{example}\label{ex:3_YD}
The $3$-dimensional Young diagram

\begin{center}

\begin{tikzpicture}[scale=0.5]
\planepartition{}{
    {4,2,1},
    {3,0,0},
    {3,0,0}}
\end{tikzpicture}
\end{center}
is flat and its tuple of partitions is given by $(3,3,\underline{4},2,1)$, that is, $\lambda=(4,3,3)$ and $\mu=(4,2,1)$. We can embed it into the plane by

\begin{center}

 \begin{ytableau}
\none & \none &4&\none&\none\\
3 &3  &&\none &\none \\
 &&&2 & \none\\
&&& & 1\\
\end{ytableau}
\end{center}

The following $3$-dimensional Young diagram is non-flat (and minimal with this property).

\begin{center}

\begin{tikzpicture}[scale=0.5]
\planepartition{}{
     {2,1},
    {1,0},
   }
\end{tikzpicture}
\end{center}
 \end{example}
 \subsection{Multi-graded nilpotent tuples}\label{ssect:multigrad}

Let $k\in\mathbf{N}$ and let $V= \bigoplus_{(x_1,..,x_k)\in\mathbf{Z}_{\geq 1}^k} V_{x_1,..,x_k}$ be an $N$-dimensional multi-graded $K$-vector space; we formally set $V_{x_1,..,x_k}:=0$ for $(x_1,..,x_k)\notin \mathbf{Z}_{\geq 1}^k$ and denote $\underline{x}:=(x_1,..,x_k)$. Let $V$ be called \textit{$k$-graded} in general, \textit{bi-graded} for $k=2$ and \textit{tri-graded} for $k=3$.

We denote by $\N(V)$ the \textit{nilpotent cone} of nilpotent operators on $V$. The \textit{$k$-fold nilpotent commuting variety} of $V$ is defined by
\[\CC^k(V):=\{(\varphi_1,...,\varphi_k)\in \N(V)^k\mid [\varphi_i, \varphi_j]=0~\mathrm{for~all}~i,j\},\]
its elements are called \textit{commuting nilpotent tuples}.

Such tuple is called \textit{graded nilpotent tuple} or \textit{$k$-graded nilpotent tuple} of  $V$, if $\varphi_i$ can be restricted via 
\[\begin{array}{rlll}
\varphi_i(\underline{x}):=\varphi_i\mid_{V_{x_1,..,x_k}}:& V_{x_1,..,x_k}& \rightarrow& V_{x_1,...,x_i-1,...,x_k}\\
&v&\mapsto& \varphi_i(v)
\end{array} \] 

We define 
 the \textit{shape} of $V$ by 
\[\sh(V):=\{(s_1,...,s_k)\mid\forall i\leq k:~ \exists p_i\in\mathbf{Z}_{\geq 1},~ s_i\leq p_i~{\rm such~that}~  V_{p_1,...,p_k}\neq 0 \}.\]
 The shape of such $k$-graded vector space $V$ is a $k$-dimensional Young diagram $Y(V)$ by definition and the special case of bi-graded nilpotent pairs coincides with the notion of  graded nilpotent pairs  for which the shape $\lambda$ of $V$ is given by an ordinary Young diagram $Y(\lambda)=Y(V)$ \cite{B4}.
 
 Let $\Lambda$ be the shape of $V$, then every graded nilpotent tuple of $V$ is called a \textit{$\Lambda$-graded nilpotent tuple}. 

\begin{example}\label{ex:multi_grad}
Let $V:= \bigoplus_{s,t,u\in\mathbf{Z}_{\geq 1}} V_{s,t,u}$ be a tri-graded $K$-vector space, such that $V_{1,1,1}=V_{1,2,2}=K$, $V_{2,1,1}=V_{1,2,1}=V_{1,1,2}=V_{2,1,2}= K^2$, $V_{2,2,1}=K^3$ and $V_{s,t,u}=0$, otherwise. Let $(\varphi^{(1)}, \varphi^{(2)}, \varphi^{(3)})$ be a graded nilpotent tuple on $V$. Then we can illustrate the latter by
\begin{center}\small\begin{tikzpicture}
\matrix (m) [matrix of math nodes, row sep=1.81em,
column sep=1.8em, text height=0.9ex, text depth=0.1ex]
{          & V_{2,1,1} & V_{2,2,1}  \\
 V_{2,1,2} & V_{1,1,1} & V_{1,2,1}  \\
 V_{1,1,2} & V_{1,2,2} &            \\
 };
\path[->]
(m-1-3) edge  (m-1-2)
(m-1-2) edge  (m-2-2)
(m-2-1) edge  (m-1-2)
(m-1-3) edge  (m-2-3)
(m-2-3) edge  (m-2-2)
(m-2-1) edge  (m-3-1)
(m-3-1) edge  (m-2-2)
(m-3-2) edge  (m-2-3)
(m-3-2) edge  (m-3-1)
(m-2-1) edge[-,dotted]  (m-2-2)
(m-1-3) edge[-,dotted]  (m-2-2)
(m-3-2) edge[-,dotted]  (m-2-2)
;\end{tikzpicture}\end{center} 

The shape of $V$ is given by the generalized Young diagram
\begin{center}

\begin{tikzpicture}[scale=0.5]
\planepartition{}{
    {2,2},
    {2,1}}
\end{tikzpicture}
\end{center}

\end{example}

\section{Connection with finite-dimensional algebras}

 \subsection{Translation}\label{ssect:translation}
Corresponding to a generalized Young diagram $\Lambda$, let us define the quiver $\Q(\Lambda)$: its vertices are given by the tuples  appearing in $\Lambda$; and there is an arrow $(x_1,...,x_n) \rightarrow (y_1,...,y_n)$ if and only if  there is exactly one index $i$, such that $y_i=x_i-1$ and $y_k=x_k$ for $k\neq i$. This arrow will be denoted by $\varphi^{(i)}_{x_1,...,x_n}$.
Especially for the bi- and tri-graded cases, the quiver $\Q(\Lambda)$ can easily be visualized by the generalized Young diagram.

Let us define the path algebra $\A(\Lambda)=K\Q(\Lambda)/I$ with relations given by $I:=I(\Lambda)$, which is the $2$-sided admissible ideal generated by all commutativity relations in the appearing squares in $\Q(\Lambda)$ which are, if defined, of the form $\varphi^{(i)}_{x_1,.x_{i},.x_j,.,x_n}\varphi^{(j)}_{x_1,.x_i,.x_j+1,.,x_n} - \varphi^{(j)}_{x_1,.x_i-1,.x_j+1,.,x_n}\varphi^{(i)}_{x_1,.,x_i,.x_j+1,.,x_n}$. Since $I$ is admissible and $\Q(\Lambda)$ is connected, $\A(\Lambda)$ is a basic, connected,  finite-dimensional $K$-algebra.
\begin{definition}\label{def:multi_staircase}
Let $\A$ be a finite-dimensional $K$-algebra. It is called a \textit{multi-graded staircase algebra} of the generalized Young diagram $\Lambda$, if  $\A\cong\A(\Lambda)$. For a $k$-dimensional Young diagram, the corresponding algebra is called $k$-graded staircase algebra, or just $k$-staircase algebra.
\end{definition}
We denote the Tits quadratic form of $\A(\Lambda)$ by $q_{\Lambda}:=q_{\A(\Lambda)}$ and the Auslander-Reiten quiver by $\Gamma(\Lambda)$. 

Let us consider an example to illustrate these definitions before discussing multi-graded staircase algebras and their properties in detail.
\begin{example}\label{ex:multi_staircase}
The quiver of the tri-graded vector space in Example \ref{ex:multi_grad} is given by
\begin{center}\small\begin{tikzpicture}[descr/.style={fill=white}]
\matrix (m) [matrix of math nodes, row sep=2.81em,
column sep=3.8em, text height=1.3ex, text depth=0.2ex]
{          & (2,1,1) & (2,2,1)  &    &  & \bullet &  \bullet\\
 (2,1,2) & (1,1,1) & (1,2,1) & \hat{=} & \bullet   &  \bullet & \bullet  \\
 (1,1,2) & (1,2,2) &      & &\bullet  & \bullet &      \\
 };
\path[->]
(m-1-3) edge node[above]{$\varphi^{(2)}_{2,2,1}$} (m-1-2)
(m-1-2) edge node[descr]{$\varphi^{(1)}_{2,1,1}$} (m-2-2)
(m-2-1) edge node[above left]{$\varphi^{(3)}_{2,1,2}$} (m-1-2)
(m-1-3) edge node[right]{$\varphi^{(1)}_{2,2,1}$} (m-2-3)
(m-2-3) edge node[descr]{$\varphi^{(2)}_{1,2,1}$} (m-2-2)
(m-2-1) edge node[left]{$\varphi^{(1)}_{2,1,2}$} (m-3-1)
(m-3-1) edge node[descr]{$\varphi^{(3)}_{1,1,2}$} (m-2-2)
(m-3-2) edge node[below right]{$\varphi^{(3)}_{1,2,2}$} (m-2-3)
(m-3-2) edge node[below]{$\varphi^{(2)}_{1,2,2}$} (m-3-1)
(m-2-1) edge[-,dotted]  (m-2-2)
(m-1-3) edge[-,dotted]  (m-2-2)
(m-3-2) edge[-,dotted]  (m-2-2)
(m-1-7) edge node[above]{$\varphi^{(2)}_{2,2,1}$} (m-1-6)
(m-1-6) edge node[descr]{$\varphi^{(1)}_{2,1,1}$} (m-2-6)
(m-2-5) edge node[above left]{$\varphi^{(3)}_{2,1,2}$} (m-1-6)
(m-1-7) edge node[right]{$\varphi^{(1)}_{2,2,1}$} (m-2-7)
(m-2-7) edge node[descr]{$\varphi^{(2)}_{1,2,1}$} (m-2-6)
(m-2-5) edge node[left]{$\varphi^{(1)}_{2,1,2}$} (m-3-5)
(m-3-5) edge node[descr]{$\varphi^{(3)}_{1,1,2}$} (m-2-6)
(m-3-6) edge node[below right]{$\varphi^{(3)}_{1,2,2}$} (m-2-7)
(m-3-6) edge node[below]{$\varphi^{(2)}_{1,2,2}$} (m-3-5)
(m-2-5) edge[-,dotted]  (m-2-6)
(m-1-7) edge[-,dotted]  (m-2-6)
(m-3-6) edge[-,dotted]  (m-2-6)
;\end{tikzpicture}\end{center}

The tri-graded staircase algebra $\A(\Lambda)$  is defined by
\[\A(\Lambda)= K\Q(\Lambda)/(
\varphi^{(2)}_{1,2,1}\varphi^{(1)}_{2,2,1} - \varphi^{(1)}_{2,1,1}\varphi^{(2)}_{2,2,1}, \varphi^{(3)}_{1,1,2}\varphi^{(1)}_{2,1,2}-\varphi^{(1)}_{2,1,1}\varphi^{(3)}_{2,1,2}, \varphi^{(3)}_{1,1,2}\varphi^{(2)}_{1,2,2}-\varphi^{(2)}_{1,2,1}\varphi^{(3)}_{1,2,2}) .\]
\end{example}

Each multi-graded nilpotent tuple  together with a multi-graded vector space $V$ of shape $\Lambda$ as in \ref{ssect:multigrad} can be considered as a representation $M:=M(\varphi_1, ...,\varphi_k,V)$ of $\A(\Lambda)$ in a natural way. We denote $\underline{\dim}V:= \underline{\dim}_{\A(\Lambda)}M$ which depends on the grading on $V$, but not on the multi-graded nilpotent tuple.

\begin{example}
Consider the setup of Example \ref{ex:multi_grad}. Then, for example, the following is an $\A(\Lambda)$-representation $M(\varphi^{(1)},\varphi^{(2)},\varphi^{(3)}, V)$:
\begin{center}\small\begin{tikzpicture}
[descr/.style={fill=white}]
\matrix (m) [matrix of math nodes, row sep=2.81em,
column sep=3.8em, text height=1.2ex, text depth=0.3ex]
{          & K^2 & K^3  \\
 K^2 & K & K^2  \\
K^2 & K &            \\
 };
\path[->]
(m-1-3) edge node[above]{$\left(\begin{array}{c}
1 0 0\\
0 1 0
\end{array}\right)$} (m-1-2)
(m-1-2) edge node[descr]{$(1,0)$} (m-2-2)
(m-2-1) edge node[above left]{$id$} (m-1-2)
(m-1-3) edge node[right]{$\left(\begin{array}{c}
1 0 1\\
1 0 0
\end{array}\right)$} (m-2-3)
(m-2-3) edge node[descr]{$(0,1)$} (m-2-2)
(m-2-1) edge node[left]{$id$} (m-3-1)
(m-3-1) edge node[descr]{$(1,0)$} (m-2-2)
(m-3-2) edge node[below right]{$(0,1)^t$} (m-2-3)
(m-3-2) edge node[below]{$(1,0)^t$} (m-3-1)
(m-2-1) edge[-,dotted]  (m-2-2)
(m-1-3) edge[-,dotted]  (m-2-2)
(m-3-2) edge[-,dotted]  (m-2-2)
;\end{tikzpicture}\end{center} 
\end{example}

 Representation-theoretically speaking, the multi-graded nilpotent tuples of $V$ are encoded (up to base change in the homogeneous coordinates, that is,  base change by a Levi subgroup) in the representation variety $R_{\underline{\dim}V}\A(\sh(V))$.  Therefore, certain criteria can be translated from the Representation Theory of finite-dimensional algebras straight away. 
\begin{theorem}\label{thm:crit_multi_grad_pairs}
Let $\Lambda$ be a generalized Young diagram. 
Modulo Levi-base change, there are only finitely many $\Lambda$-graded nilpotent tuples  if and only if  $\A(\Lambda)$ is of finite representation type.
Otherwise, there are at most one-parameter families of non-decomposable $\Lambda$-graded nilpotent tuples if and only if $\A(\Lambda)$ is tame.
\end{theorem}
This criterion shows that our Main Theorem is a direct consequence of our Classification Theorem and motivates the classification of representation types of multi-graded staircase algebras. 

\begin{lemma}\label{lem:grad_fin_dim}
The number of multi-graded nilpotent tuples (up to Levi-base change) of the fixed multi-graded vector space $V$ is finite if and only if $R_{\underline{\dim}V}\A(\sh(V))$ admits only finitely many $\GL_{\underline{\dim}V}$-orbits. In particular, this is the case if and only if the number of isomorphism classes of representations in $\rep\A(\sh(V))(\underline{\dim} V)$ is finite.
\end{lemma}
We, thus, study multi-graded staircase algebras to prove our Classification Theorem which directly yields our Main Theorem.

\subsection{General properties}\label{ssect:properties}
From now on, let $\Lambda$ be a $k$-dimensional Young diagram for some integer $k$. For $(i_1,...,i_k)\in \Q(\Lambda)_0$, let $S(i_1,....,i_k)$ be the standard simple representation at the vertex $(i_1,...,i_k)$ of $\A(\Lambda)$. 
The (isomorphism classes of the) projective indecomposables of $\A(\Lambda)$ are parametrized by $P(i_1,...,i_k)$, $(i_1,...,i_k)\in \Q(\Lambda)_0$, which are given by
 \[P(i_1,...,i_k)_{l_1,...,l_k}=\left\lbrace
\begin{array}{ll}
K & {\rm if}~ l_j\leq i_j ~\forall j, \\ 
0 & {\rm otherwise}.
                                                               \end{array}
 \right. \]
together with identity and zero maps, accordingly.
\begin{proposition}\label{prop:properties}
The algebra $\A(\Lambda)$ is triangular, fulfills the separation condition and is strongly simply connected. 
\end{proposition}
\begin{proof} Since $\Q(\Lambda)$ does not contain oriented cycles, $\A(\Lambda)$ is triangular.

The vector spaces $P(i_1,...,i_k)_{l_1,...,l_k}$ are one-dimensional for every $(l_1,...,l_k)\in\Q(\Lambda)_0$, thus, $\A(\Lambda)$ fulfills the separation condition. Since the latter is the case for every convex subcategory of $\A(\Lambda)$, the algebra is strongly simply connected. 
\end{proof}
These properties can directly  be translated and restricted to convex subcategories. We obtain induced simples $S(i_1,...,i_k)$ and projective indecomposables $P(i_1,...,i_k)$ for every vertex. Also, for the same reasons as in the proof of Proposition \ref{prop:properties}, every convex subcategory is triangular, fulfills the separation condition and is strongly simply connected.

\subsection{Reductions}\label{ssect:reductions}
The following reductions will help to examine the representation types of all multi-graded staircase algebra.
\begin{lemma}\label{lem:reduct_subcat}
 Let $\A$ be a convex subcategory of $\A'$. Then
 \begin{enumerate}
\item ... if $\A$ is tame, then $\A'$ is tame or wild.
 \item ... if $\A$ is wild, then $\A'$ is wild.
\item ... if $\A'$ has finite representation type, $\A$ has finite representation type.
\item ... if $\A'$ is tame, then $\A$ is tame or of finite representation type.
\end{enumerate}
 In particular, if $\Lambda \subseteq \Lambda'$ are two generalized Young diagrams, then these facts hold true for $\A=\A(\Lambda)$ and $\A'=\A(\Lambda')$.
\end{lemma}
\begin{proof}
 The claim follows from general Representation Theory of quivers with relations \cite{ASS} by restricting in a natural way or expanding with zeros. 
\end{proof}

\begin{definition}\label{def:correlation}
We say that two algebras $\A=K\Q/I$ and $\A'=K\Q'/I'$ correlate, if the quivers $\Q$ and $\Q'$ coincide modulo turning around of arrows and if the ideals coincide in the same way. Two correlating algebras are denoted by $\A\sim\A'$. In particular, if $\A(\Lambda)\sim\A(\Lambda')$, we denote $\Lambda\sim \Lambda'$. 
\end{definition}
\begin{example}\label{ex:correlation}
The algebras of the quivers
\[\begin{tikzpicture}
\matrix (m) [matrix of math nodes, row sep=0.8em,
column sep=0.8em, text height=0.2ex, text depth=0.1ex]
{&& &  \bullet \\
  & \bullet & \bullet & \bullet \\
 \bullet &\bullet & \bullet &\\
 };
\path[->]
(m-2-2) edge  (m-2-3)
(m-2-3) edge  (m-2-4)
(m-2-2) edge (m-3-2)
(m-2-3) edge (m-3-3)
(m-1-4) edge (m-2-4)
(m-3-1) edge (m-3-2)
(m-3-2) edge  (m-3-3);
\path[-]
(m-2-2) edge[dotted]  (m-3-3);
\end{tikzpicture} ~~~and~~~ \begin{tikzpicture}
\matrix (m) [matrix of math nodes, row sep=0.8em,
column sep=0.8em, text height=0.2ex, text depth=0.1ex]
{&& \bullet\\
& &  \bullet \\
  & \bullet & \bullet  \\
 \bullet &\bullet & \bullet \\
 };
\path[->]
(m-1-3) edge (m-2-3)
(m-2-3) edge  (m-3-3)
(m-3-2) edge  (m-3-3)
(m-3-2) edge (m-4-2)
(m-3-3) edge (m-4-3)
(m-4-1) edge (m-4-2)
(m-3-2) edge  (m-3-3)
(m-4-2) edge  (m-4-3);
\path[-]
(m-3-2) edge[dotted]  (m-4-3);
\end{tikzpicture}\]
correlate. 
\end{example}

\begin{lemma}\label{lem:reduct_orientation}
If $\A$ and $\A'$ are multi-graded staircase algebras and $\A\sim \A'$, then the representation types of $\A$ and $\A'$ coincide.
\end{lemma}
\begin{proof}
Since $\A$ and $\A'$ are strongly simply connected, we know that their representation types only depend on the quadratic form. It, thus, does not depend on the orientation of the corresponding quiver.
\end{proof}
In particular, we know that the representation type is invariant under certain symmetries.

 In the next sections, we classify all representation types of multi-graded staircase algebras. In more detail, in Section \ref{sect:bi_graded} we cite the classification of bi-graded staircase algebras of \cite{B4}, in Section \ref{sect:tri_graded} we classify the representation types of all tri-graded staircase algebras and in Section \ref{sect:high_grad}, we finish the proof of our Classification Theorem by classifying all higher-graded staircase algebras.
 
We define a \textit{proper} $k$-graded staircase algebra  to be a  $k$-graded staircase algebra $\A(\Lambda)$ which does not correlate to a $k-1$-graded staircase algebra. This is the case iff $\Lambda$ contains the minimal $k$-dimensional Young Diagram \[\Lambda_{\min}:=\{(x_1,...,x_k)\mid ~x_i= 2~{\rm for~at~most~one}~i\}\subset \{1,2\}^k.\]

\section{Bi-graded nilpotent pairs}\label{sect:bi_graded}
The "smallest" case of multi-graded nilpotent tuples is the case of graded nilpotent pairs, which correspond by Subsection \ref{ssect:translation} to bi-graded staircase algebras, which we only call \textit{staircase algebras} as in \cite{B4}.

A complete classification of representation types of staircase algebras is available in \cite{B4}. Since every staircase algebra is defined by its corresponding shape which is a Young diagram, that is, partition, it suffices to classify these partitions.
\begin{theorem}(\cite[Theorem 4.6]{B4})\label{thm:stair_rep_type}
A staircase algebra $\A(\lambda)$ is 
\begin{itemize}
\item[(a)] representation-finite if and only if one of the following holds true:
\begin{enumerate}
 \item\label{1stcase} $\lambda\in\{(n),~ (1^k,n-k),~ (2,n-2),~ (1^{n-4},2^2)\}$ for some $k\leq n$,
\item\label{2ndcase}  $n\leq 8$ and $\lambda\notin\{(1,3,4),~ (2,3^2),~ (1,2^2,3),~ (1^2,2,4)\}$.
\end{enumerate}
\item[(b)] tame concealed if and only if $\lambda$ comes up in the following list:\\ $(3,6)$, $(1,2,6)$, $(1,3,4)$, $(2^2,5)$, $(1^2,2,4)$, $(1,2^2,3)$, $(1^3,3^2)$, $(1^3,2^3)$, $(1^4,2,3)$.
\item[(c)] tame, but not tame concealed if and only if $\lambda$ comes up in the following list:\\ $(4,5)$, $(5^2)$, $(1,4^2)$, $(2,3^2)$, $(3^3)$, $(2^3,3)$, $(1,2^4)$, $(2^5)$.
\end{itemize} 
 Otherwise, $\A(\lambda)$ is of wild representation type. 
\end{theorem}
As in \cite{B4}, this yields the following finiteness criterion for graded nilpotent pairs.
\begin{theorem}(\cite[Theorem 6.1]{B4})\label{thm:gradedLambda}
The number of $\lambda$-graded nilpotent pairs (up to Levi-base change) is finite if and only if $\lambda$ appears in Theorem \ref{thm:stair_rep_type} (a).
\end{theorem}
Further criteria about the number of graded nilpotent pairs of fixed bi-graded vector spaces are available by \cite{B4}.
\section{Tri-graded nilpotent tuples}\label{sect:tri_graded}
We now discuss tri-graded nilpotent tuples which correspond to tri-graded staircase algebras via Subsection \ref{ssect:translation}.

A proper tri-graded staircase algebra $\A$ is called \textit{flat} if it correlates to a tri-graded staircase algebra $K\Q(\Lambda)/I$, where $\Lambda$ is a flat tri-graded Young diagram, that is, its quiver can be  embedded into the plane. In particular, as we have seen in Subsection \ref{ssect:gener_YD}, every flat tri-graded staircase algebra $\A$ is given by two decreasing partitions $\lambda=(\lambda_1,...,\lambda_k)$ and $\mu=(\mu_1,...,\mu_l)$, such that $\lambda_1=\mu_1=:x$ and such that the quiver of $\A$ is obtained by gluing the two quivers of the staircase algebras $\A(\lambda)$ and $\A(\mu)$ together at their common column of $x$ vertices. We denote $\A=:\A(\lambda,\mu)$ and $(\lambda,\mu)=(\lambda_k,...\lambda_2,\underline{x},\mu_2,...,\mu_l)$.  

We define a componentwise ordering $(\lambda,\mu)\leq_c (\lambda',\mu') $ on these double partitions by the condition that $\lambda\leq \lambda'$ is a subpartition and $\mu\leq \mu'$ is a subpartition (speaking in Young diagrams as in Subsection \ref{ssect:gener_YD}, they are subsets). If $(\lambda,\mu)\leq_c (\lambda',\mu')$, then the statements of Lemma \ref{lem:reduct_subcat} hold true for $\A(\lambda,\mu)$ which is a convex subcategory of $\A(\lambda',\mu')$.
\begin{example}\label{ex:flat}
For the flat $3$-dimensional Young diagram from Example \ref{ex:3_YD}, the partitions are given by $\lambda=(4,2,1)$ and $\mu=(4,3,3)$; and the algebra $\A(\lambda,\mu)=\A(1,2,\underline{4},3,3)$ is defined by the quiver
\begin{center}
\begin{tikzpicture}
\matrix (m) [matrix of math nodes, row sep=0.7em,
column sep=0.7em, text height=0.79ex, text depth=0.1ex]
{       &         & \bullet  &         &\\
        &         & \bullet  & \bullet & \bullet\\
        & \bullet & \bullet  & \bullet & \bullet\\
\bullet & \bullet & \bullet  & \bullet & \bullet\\
 };
\path[-]
(m-1-3) edge  (m-2-3)
(m-2-3) edge  (m-3-3)
(m-2-4) edge  (m-2-3)
(m-2-5) edge   (m-2-4)
(m-2-4) edge   (m-3-4)
(m-2-5) edge   (m-3-5)
(m-3-2) edge  (m-3-3)
(m-3-4) edge   (m-3-3)
(m-3-5) edge   (m-3-4)
(m-3-2) edge   (m-4-2)
(m-3-3) edge   (m-4-3)
(m-3-4) edge   (m-4-4)
(m-3-5) edge   (m-4-5)
(m-4-1) edge   (m-4-2)
(m-4-2) edge   (m-4-3)
(m-4-4) edge   (m-4-3)
(m-4-5) edge   (m-4-4);
\path[-]
(m-3-2) edge[dotted]    (m-4-3)
(m-2-4) edge[dotted]    (m-3-3)
(m-2-5) edge[dotted]    (m-3-4)
(m-3-4) edge[dotted]    (m-4-3)
(m-3-5) edge[dotted]    (m-4-4)
;\end{tikzpicture}
\end{center}
\end{example}
Our aim in the next two subsections is to prove the following classification of representation types of tri-graded staircase algebras.
\begin{theorem}\label{thm:tri_stair_rep_type}
A proper tri-graded staircase algebra $\A(\Lambda)$ is 
\begin{itemize}
\item[(a)] representation-finite if and only if $\A(\Lambda)=\A(\lambda,\mu)$ is flat and 
$(\lambda,\mu)\leq_c (\lambda',\mu')$ appears in Lemma \ref{lem:3flat_rep_type} (a),
\item[(b)] tame concealed if and only if $\A(\Lambda)=\A(\lambda,\mu)$ is flat and 
$(\lambda,\mu)$ appears in Lemma \ref{lem:3flat_rep_type} (b),
 
\item[(c)] tame, but not tame concealed if and only if
either $\Lambda$ equals 
\begin{center}
\begin{tikzpicture}[scale=0.5]
\planepartition{}{
    {2,2},
    {2,1}}
\end{tikzpicture}
\end{center}
or if $\A(\Lambda)=\A(\lambda,\mu)$ is flat and
$(\lambda,\mu)$ appears in Lemma \ref{lem:3flat_rep_type} (c).
\end{itemize} 
 Otherwise, $\A(\Lambda)$ is of wild representation type. 
\end{theorem}
The explicit flat cases are listed in Appendix \ref{app:case_study} and we split up our examination into two cases: flat and non-flat tri-graded staircase algebras.

\subsection{Non-flat tri-gradings}\label{ssect:tri_non-flat}

We first discuss the representation types of non-flat tri-graded staircase algebras. Therefore, we define one particular algebra $\A(3)$ by the quiver with relations
\begin{center}\small\begin{tikzpicture}
\matrix (m) [matrix of math nodes, row sep=0.8em,
column sep=0.8em, text height=0.8ex, text depth=0.1ex]
{          & \bullet_b & \bullet_e  \\
 \bullet_g & \bullet_a & \bullet_c  \\
 \bullet_d & \bullet_f &            \\
 };
\path[->]
(m-1-3) edge  (m-1-2)
(m-1-2) edge  (m-2-2)
(m-2-1) edge  (m-1-2)
(m-1-3) edge  (m-2-3)
(m-2-3) edge  (m-2-2)
(m-2-1) edge  (m-3-1)
(m-3-1) edge  (m-2-2)
(m-3-2) edge  (m-2-3)
(m-3-2) edge  (m-3-1)
(m-2-1) edge[-,dotted]  (m-2-2)
(m-1-3) edge[-,dotted]  (m-2-2)
(m-3-2) edge[-,dotted]  (m-2-2)
;\end{tikzpicture}\end{center} 
\begin{lemma}\label{lem:3non-flat_rep_type}
Let $\A$ be a non-flat tri-graded staircase algebra. Then $\A$ is of infinite representation type. In more detail, $\A$ is tame non-concealed if and only if $\A=\A(3)$ and wild otherwise.
\end{lemma}
\begin{proof}
Every proper tri-graded non-flat staircase algebra contains the algebra $\A(3)$ as a subcategory. This algebra is tame, since it does not contain a hypercritical algebra \cite{Un}. By Theorem \ref{thm:crit_tame}, we can also show tameness straight away, since the corresponding quadratic form $q\coloneqq q_{\A(3)}$ is weakly non-negative (we stick to the notation used before):
\[\begin{array}{ll}
q(a,b,c,d,e,f,g)&= (\frac{a}{2}-b+\frac{e}{2}+ \frac{g}{2})^2
+ (\frac{a}{2}-c+\frac{e}{2}+ \frac{f}{2})^2+ (\frac{a}{2}-d+\frac{f}{2}+ \frac{g}{2})^2\\
& ~~~~ + \frac{a^2}{4}+ (\frac{e}{2}-\frac{f}{2})^2+ (\frac{e}{2}-\frac{g}{2})^2+ (\frac{f}{2}-\frac{g}{2})^2\\
& \geq 0.
\end{array} \]
The algebra $\A(3)$ is not tame concealed by \cite{HaVo} (see Subsection \ref{ssect:theory}), since it contains an algebra of extended Dynkin type $\widetilde{A}_6$ as a convex subcategory which has the minimal nullroot
\begin{center}\small\begin{tikzpicture}
\matrix (m) [matrix of math nodes, row sep=0.8em,
column sep=0.8em, text height=0.8ex, text depth=0.1ex]
{          & \bullet_1 & \bullet_1  \\
 \bullet_1 & \bullet_0 & \bullet_1  \\
 \bullet_1 & \bullet_1 &            \\
 };
\path[->]
(m-1-3) edge  (m-1-2)
(m-1-2) edge  (m-2-2)
(m-2-1) edge  (m-1-2)
(m-1-3) edge  (m-2-3)
(m-2-3) edge  (m-2-2)
(m-2-1) edge  (m-3-1)
(m-3-1) edge  (m-2-2)
(m-3-2) edge  (m-2-3)
(m-3-2) edge  (m-3-1)
(m-2-1) edge[-,dotted]  (m-2-2)
(m-1-3) edge[-,dotted]  (m-2-2)
(m-3-2) edge[-,dotted]  (m-2-2)
;\end{tikzpicture}\end{center} 
Every remaining proper tri-graded non-flat staircase algebra is wild by Lemma \ref{lem:reduct_subcat}, since the algebra $\A$ of the following quiver with relations
\begin{center}
\begin{tikzpicture}
\matrix (m) [matrix of math nodes, row sep=0.8em,
column sep=0.8em, text height=0.8ex, text depth=0.1ex]
{          & \bullet_2 & \bullet_2  \\
 \bullet_2 & \bullet_0 & \bullet_2 &\bullet_1  \\
 \bullet_2 & \bullet_2 &            \\
 };
\path[->]
(m-1-3) edge  (m-1-2)
(m-1-2) edge  (m-2-2)
(m-2-1) edge  (m-1-2)
(m-1-3) edge  (m-2-3)
(m-2-3) edge  (m-2-2)
(m-2-1) edge  (m-3-1)
(m-3-1) edge  (m-2-2)
(m-3-2) edge  (m-2-3)
(m-3-2) edge  (m-3-1)
(m-2-4) edge  (m-2-3)
(m-2-1) edge[-,dotted]  (m-2-2)
(m-1-3) edge[-,dotted]  (m-2-2)
(m-3-2) edge[-,dotted]  (m-2-2)
;\end{tikzpicture}
\end{center}
is of wild representation type (the sketched dimension vector $\df$ fulfills $q_{\A}(\df)=-1$ and Corollary \ref{cor:crit_wild} yields wildness, thus). 
\end{proof}
For tri-graded nilpotent tuples, we directly get the following criterion.
 \begin{proposition}
Let $V$ be a proper tri-graded non-flat vector space. Then there are (up to Levi-base change) infinitely many tri-graded nilpotent tuples of $V$.
\end{proposition}
 \begin{proof}
Since there are infinitely many isomorphism classes of representations of dimension vector  \begin{center}\small\begin{tikzpicture}
\matrix (m) [matrix of math nodes, row sep=0.2em,
column sep=0.2em, text height=0.2ex, text depth=0.1ex]
{          & 1 & 1  \\
 1 & 0 & 1  \\
 1 & 1 &     \\ 
 };
\end{tikzpicture},\end{center}
the claim follows from Lemma \ref{lem:grad_fin_dim}.
\end{proof}

\subsection{Flat tri-gradings}\label{ssect:tri_flat}
It remains to  classify the representation type of each flat tri-graded staircase algebra.

Let $\lambda=(\lambda_1,...,\lambda_k)$ and $\mu=(\mu_1,...,\mu_l)$ be two decreasing partitions, such that $\lambda_1=\mu_1=:x$. These yield a flat tri-graded staircase algebra $\A(\lambda,\mu)$ and our aim is to decide: Is this algebra of finite representation type? Note that clearly the representation type of $\A(\lambda,\mu)$ and $\A(\mu,\lambda)$ coincides.

 The number $k+l-1$ is called the \textit{ground length} of $\A(\lambda,\mu)$; the number $k$ is called \textit{wall}.
 
 \begin{example}
The algebra $\A(\lambda,\mu)$, where $\lambda=(4,2,1)$ and $\mu=(4,3,3)$ as in Example \ref{ex:flat} is given by the following quiver with relations.
\begin{center}
\begin{tikzpicture}
\matrix (m) [matrix of math nodes, row sep=0.7em,
column sep=0.7em, text height=0.79ex, text depth=0.1ex]
{ &      &         & \bullet  &         &\\
  &      &         & \bullet  & \bullet & \bullet\\
  &      & \bullet & \bullet  & \bullet & \bullet\\
&\bullet & \bullet & \bullet  & \bullet & \bullet\\
(\lambda,\mu)=&(1 & 2 & \underline{4}  & 3 & 3)\\ };
\path[->]
(m-1-4) edge  (m-2-4)
(m-2-4) edge  (m-3-4)
(m-2-5) edge  (m-2-4)
(m-2-6) edge   (m-2-5)
(m-2-5) edge   (m-3-5)
(m-2-6) edge   (m-3-6)
(m-3-3) edge  (m-3-4)
(m-3-5) edge   (m-3-4)
(m-3-6) edge   (m-3-5)
(m-3-3) edge   (m-4-3)
(m-3-4) edge   (m-4-4)
(m-3-5) edge   (m-4-5)
(m-3-6) edge   (m-4-6)
(m-4-2) edge   (m-4-3)
(m-4-3) edge   (m-4-4)
(m-4-5) edge   (m-4-4)
(m-4-6) edge   (m-4-5);
\path[-]
(m-3-3) edge[dotted]    (m-4-4)
(m-2-5) edge[dotted]    (m-3-4)
(m-2-6) edge[dotted]    (m-3-5)
(m-3-5) edge[dotted]    (m-4-4)
(m-3-6) edge[dotted]    (m-4-5)
;\end{tikzpicture}.
\end{center}
The algebra has ground length $5$ which is given by the number of entries of 
$(\lambda,\mu)$. The quiver also illustrates that the wall is $3$; this is the length of $\lambda$ and can be thought of as the line of vertices where all paths end.
 \end{example}


\begin{lemma}\label{lem:3flat_rep_type}
A proper flat tri-graded staircase algebra $\A(\lambda,\mu)$ is  
\begin{enumerate}
\item[(a)]
of finite representation type if and only if there is $(\lambda',\mu')$, such that $(\lambda,\mu)\leq_c (\lambda',\mu')$ or $(\mu,\lambda)\leq_c (\lambda',\mu')$ and such that there is some $x\in\mathbf{N}$ for which $(\lambda',\mu')$ comes up in the following list:
\[\begin{array}{l}
(1,\underline{5},2), (4,\underline{4},1), (1,\underline{5},1^2), (2,\underline{2},2^2), (1,\underline{2},2^3),(3,\underline{3},1^3), (1,\underline{3},1^4), \\
 (1,\underline{2},2,1^3), (2^2,\underline{2},1^3), (1^2,\underline{2},1^4), (1,\underline{x},1),
 (1,\underline{2},1^x), (2,\underline{2},1^x),
\end{array}
  \]

\item[(b)] tame concealed if and only if $(\lambda,\mu)$ or $(\mu,\lambda)$ comes up in the following list:
\[\begin{array}{l}
(2,\underline{3},2) , (3,\underline{5},1), (2,\underline{6},1), (1,\underline{3},2,1),(2,\underline{4},1^2),(1,\underline{6},1^2), (1,\underline{4},1^3),  \\
 (1,2,\underline{2},2,1),  (1^2,\underline{3},1^2),  (1,2^2,\underline{2},2),   (1,\underline{2},2^2,1^2),   (1^2,\underline{2},2,1^2), (2,\underline{2},2,1^3),  \\
 (2,\underline{3},1^4), (1,\underline{2},2,1^4), (1,2,\underline{2},1^4), (1,\underline{3},1^5), (1^3,\underline{2},1^3), (1^2,\underline{2},1^5),
\end{array} \]
\item[(c)] tame non-concealed if and only if $(\lambda,\mu)$ or $(\mu,\lambda)$ comes up in the following list:
\[\begin{array}{l}
(2,\underline{3},3) ,(3,\underline{3},3), (4,\underline{5},1), (1,\underline{3},2^2), (1,\underline{3},3,1), (3,\underline{4},1^2),(2,\underline{2},2^3),\\
 (2^2,\underline{2},2,1), (2^2,\underline{2},2^2), (1,\underline{2},2^3,1),(1^2,\underline{2},2^2,1),  (3,\underline{3},1^4), (2^2,\underline{2},1^4).
\end{array} \]

\end{enumerate}
Otherwise, $\A(\lambda,\mu)$ is of wild representation type. 

\end{lemma}
The cases are illustrated in Appendix \ref{app:case_study}; for the tame concealed cases, the numbering of the BHV-list is added there, as well.
\begin{proof}
To begin with, we mention that for each integer $x$, the algebra $\A(2,\underline{2},1^x)$ is of finite representation type; the Auslander-Reiten quiver of the case $\A(2,\underline{2},1^4)$ can be obtained by knitting techniques and is depicted in \ref{app:arq_221111} and its generalization is easily seen.

We structure the classification by the ground length of $\A(\lambda,\mu)$.

\textit{Ground length $3$}: The wall is $2$ and and we  begin by showing that $\A(\lambda,\mu)$  is of finite representation type if (up to symmetry by Lemma \ref{lem:reduct_orientation})
 $(\lambda,\mu)\leq_c(\lambda',\mu')$, where  $(\lambda',\mu')\in\{ (1,\underline{x},1), (2,\underline{2},2), (1,\underline{5},2), (4,\underline{4},1)\}$ for some integer $x$.
By reduction via Lemma \ref{lem:reduct_subcat}, it suffices to show finiteness for the maximal cases:  
The algebra $\A(1,\underline{x},1)$ is of finite representation type for arbitrary $x$, since  the algebra is of Dynkin type $D_{x+2}$. The algebras
$\A(2,\underline{2},2)$ and $\A(2,\underline{5},1)$ are of finite representation type as subcategories of tame concealed algebras (in more detail, numbers $92$ and $20$ of the BHV-list \cite{GaRo}).
And $\A(4,\underline{4},1)$ is  correlated to $A(2^3,\underline{2},1)$ which is of finite representation type, since the Auslander Reiten quiver is finite (see \ref{app:arq_12222}).

The algebras $\A(2,\underline{3},2)$, $\A(2,\underline{6},1)$ and $\A(3,\underline{5},1)$ are tame concealed by the BHV-list (numbers $92$, $20$ and $86$ \cite{GaRo}). Furthermore, the cases $\A(2,\underline{3},3)$, $\A(3,\underline{3},3)$ and $\A(4,\underline{5},1)$ are not of finite representation type and not tame concealed \cite{HaVo} (see Subsection \ref{ssect:theory}), since they contain tame concealed subcategories. They are not wild and, thus, tame non-concealed, by Theorem \ref{thm:crit_tame}, since they do not contain any hypercritical subcategory \cite{Un}.

Every remaining case is of wild representation type by Lemma \ref{lem:reduct_orientation}, because $\A(2,\underline{4},2)$ and $\A(2,\underline{7},1)$ are hypercritical and $\A(3,\underline{6},1)$ contains the hypercritical subcategory $\A(2,\underline{5},1^2)$.

\textit{Ground length $4$}: Without loss of generality, by Lemma \ref{lem:reduct_orientation}, we choose the wall to be $2$. The algebra $\A(\lambda,\mu)$ is of finite representation type if
$(\lambda,\mu)\leq_c (\lambda',\mu')$, where $(\lambda',\mu')\in \{(1,\underline{5},1^2), (3,\underline{3},1^2), (2,\underline{2},2^2)\}$.
These cases are finite by Lemma \ref{lem:reduct_subcat}, since the maximal cases are finite: $\A(1,\underline{5},1^2)$ is of Dynkin type $E_8$ and the algebra $\A(3,\underline{3},1^2)$ is an idempotent quotient of the algebra $\A(3,\underline{3},1^3)$ which is of finite representation type. The algebra $\A(2,\underline{2},2^2)$ is of finite representation type as a subcategory of the tame concealed algebra number $137$ \cite{GaRo}.

There are three tame concealed cases $\A(1,\underline{3},2,1)$, $\A(2,\underline{4},1^2)$ and $\A(1,\underline{6},1^2)$ which correspond to numbers $11$ and $12$ in the BHV-list \cite{GaRo} and to the Euclidean quiver of type $\widetilde{E_8}$. Furthermore, the cases 
$\A(1,\underline{3},3,1)\sim \A(1,\underline{3},2^2)$ and $\A(3,\underline{4},1^2)$ are not finite, as they contain subcategories of tame concealed types $\A(1,\underline{3},2,1)$ or $\A(2,\underline{4},1^2)$ and not tame concealed for the same reason \cite{HaVo} (see Subsection \ref{ssect:theory}). They are not wild by Theorem \ref{thm:crit_tame}, since they do not contain any hypercritical algebra as a subcategory. Thus, they are tame non-concealed.

Every remaining case is of wild representation type, since the algebras $\A(2,\underline{3},2,1)$, $\A(1,\underline{4},2,1)$ $\A(2,\underline{5},1^2)$ and $\A(1,\underline{7},1^2)$ are hypercritical by \cite{Un}.

\textit{Ground length $5$}: By symmetry \ref{lem:reduct_orientation}, two cases appear: either the wall equals $2$ or $3$.

Let the wall be $2$. We show that $\A(\lambda,\mu)$ is of finite representation type if 
$(\lambda,\mu)\leq_c(\lambda',\mu')$, such that $(\lambda',\mu')\in \{ (2,\underline{2},2,1^2),  (1,\underline{2},2^3), (3,\underline{3},1^3)\}$. These cases are finite by Lemma \ref{lem:reduct_subcat}, since the Auslander-Reiten quivers of  $\A(3,\underline{3},1^3)$ and $\A(1,\underline{2},2^3)$ are finite and depicted in \ref{app:arq_33111} and \ref{app:arq_12222}; and since the algebra $\A(2,\underline{2},2,1^2)$ is of finite representation type as a subcategory of a tame concealed algebra (namely, of number $95$ \cite{GaRo}).

The algebra $\A(2,\underline{2},2^2,1)$ is tame concealed as number $137$ of the BHV-list \cite{GaRo} and $\A(1,\underline{4},1^3)$ is tame concealed of Euclidean type $\widetilde{E_7}$. Furthermore, the algebra $\A(2,\underline{2},2^3)$ is tame non-concealed as it contains the tame concealed subcategory 
$\A(2,\underline{2},2^2,1)$ \cite{HaVo} (see Subsection \ref{ssect:theory}), but no hypercritical subcategory \cite{Un}. Every remaining case is wild, since the algebras $\A(1,\underline{3},2,1^2)$ and $\A(2,\underline{4},1^3)$ are hypercritical \cite{Un}.

Let the wall be $3$. Then $\A(\lambda,\mu)$ is of finite representation type if and only if 
$(\lambda,\mu)\leq_c (2^2,\underline{2},1^2)$ by Lemma \ref{lem:reduct_orientation}: the algebra $\A(2^2,\underline{2},1^2)$ is of finite representation type as a subcategory of the tame concealed algebra number $86$ \cite{GaRo}. 
The only tame concealed cases are $\A(1,2,\underline{2},2,1)$ and $\A(1^2,\underline{3},1^2)$, the former is number $93$ of the BHV-list \cite{GaRo}, the latter is of Euclidean type $\widetilde{E}_6$. There are two tame non-concealed algebras, namely $\A(2^2,\underline{2},2,1)$ and $\A(2^2,\underline{2},2^2)$: these contain the tame concealed subcategory $\A(1,2,\underline{2},2,1)$ \cite{HaVo} (see Subsection \ref{ssect:theory}), but no hypercritical subcategory \cite{Un}.
Every remaining case is of wild representation type, since $\A(1,2,\underline{3},1^2)$ is hypercritical by \cite{Un}.

\textit{Ground length $6$}: Again, there are (up to symmetry \ref{lem:reduct_orientation}) two cases to consider: the wall equals $2$ or $3$.\\
First, let the wall be $3$. Then $\A(\lambda,\mu)$ is of finite representation type if  $(\lambda,\mu)\leq_c (2^2,\underline{2},1^3)$, since $\A(2^2,\underline{2},1^3)\sim \A(3,\underline{3},1^3)$ has a finite Auslander-Reiten quiver \ref{app:arq_33111}. There is one tame concealed algebra, namely $\A(1^2,\underline{2},2,1^2)$ which corresponds to number $12$ in the BHV-list \cite{GaRo}. Furthermore, the algebra $\A(1^2,\underline{2},2^2,1)$ is tame non-concealed: representation-infinite as it contains the tame concealed algebra $\A(1^2,\underline{2},2,1^2)$ as a subcategory, thus also non-concealed by \cite{HaVo} (see Subsection \ref{ssect:theory}) and not wild by Theorem \ref{thm:crit_tame}, since it does not contain any hypercritical subcategory. Every remaining case is of wild representation type by reduction \ref{lem:reduct_subcat}, since the algebras $\A(1,2,\underline{2},2,1^2)$ and and $\A(1^2,\underline{3},1^3)$ are hypercritical \cite{Un}.

Let the wall be $2$, then  $\A(\lambda,\mu)$ is of finite representation type if $(\lambda,\mu)\leq_c (\lambda',\mu')$, where $(\lambda',\mu')\in\{(2,\underline{2},1^4), (1,\underline{2},2,1^3), (1,\underline{3},1^4) \}$.
These cases are finite by \ref{lem:reduct_subcat}: $\A(2,\underline{2},1^4)\sim \A(3,\underline{3},1^3)$ has been seen to be of finite representation type before (see Appendix \ref{app:arq_33111} for its finite Auslander-Reiten quiver) and the algebra $\A(1,\underline{2},2,1^3)$ is of finite representation type as a subcategory of the tame concealed algebra number $20$ \cite{GaRo}. The algebra $\A(1,\underline{3},1^4)$ is of Dynkin type $E_8$. There are three tame concealed cases, namely $\A(2,\underline{2},2,1^3)$, $\A(1,\underline{2},2^2,1^2)$ and $\A(2,\underline{3},1^4)$ which correspond to numbers $95$, $86$ and $17$ of the BHV-list \cite{GaRo}. The algebras $\A(1,\underline{2},2^3,1)$ and $\A(3,\underline{3},1^4)$ are tame non-concealed: they each contain tame concealed subcategories (namely, $\A(1,\underline{2},2^2,1^2)$ and $\A(2,\underline{3},1^4)$), but no hypercritical subcategory \cite{Un}. 
Every remaining case is wild, since $\A(1,\underline{4},1^4)$ is hypercritical and 
$\A(1,\underline{3},2,1^3)$ contains the hypercritical algebra $\A(1,\underline{3},2,1^2)$.

\textit{Ground length $7$}: In this case, the wall can be assumed to be $2$, $3$ or $4$ (up to symmetry by Lemma \ref{lem:reduct_orientation}).

Let the wall be $2$: The only finite case $\A(2,\underline{2},1^5)$ has been considered before. There are two tame concealed cases: $\A(1,\underline{2},2,1^4)$ as number $20$ of the BHV-list \cite{GaRo} and $\A(1,\underline{3},1^5)$ which is of Euclidean type $\widetilde{E_8}$.  Every remaining case is of wild representation type by Lemma \ref{lem:reduct_subcat}, since $\A(2,\underline{2},2,1^4)$ is hypercritical \cite{Un} and $\A(1,\underline{2},2^2,1^3)$, $\A(2,\underline{3},1^5)$ and $\A(1,\underline{3},2,1^4)$ contain hypercritical subcategories.

Let the wall be $3$:
The only finite case $\A(1^2,\underline{2},1^4)$  is of Dynkin type $E_8$. There is one tame concealed algebra $\A(1,2,\underline{2},1^4)$ (namely, number $17$ of the BHV-list \cite{GaRo}) and one tame non-concealed algebra $\A(2^2,\underline{2},1^4)$ which contains the mentioned tame concealed subcategory, but no hypercritical subcategory (Theorem \ref{thm:crit_tame}) \cite{Un}. Every remaining case is wild, since $\A(1^2,\underline{2},2,1^3)$  and $\A(1^2,\underline{3},1^4)$ are wild (the former is hypercritical and that latter contains the hypercritical $\A(1^2,\underline{3},1^3)$, see ground length $6$).

Let the wall be $4$:  There is no case of finite representation type, since the algebra $\A(1^3,\underline{2},1^3)$ is tame concealed of Euclidean type $\widetilde{E_7}$. Every remaining case is of wild representation type, since $\A(1^2,2,\underline{2},1^3)$ is hypercritical and $\A(1^3,\underline{3},1^3)$ contains the hypercritical subcategory $\A(1^3,\underline{3},1^2)$.

\textit{Ground length $\geq 8$}:\\
 The only cases of finite representation type are $\A(2,\underline{2},1^x)$ (see above).
Every remaining algebra $\A(\lambda,\mu)$ is of wild representation type, since it contains one of the hypercritical subcategories $\A(1,\underline{2},2,1^5)$, $\A(1,\underline{3},1^6)$, $\A(1^3,\underline{2},1^4)$, $\A(1,2,\underline{2},1^5)$ or  $\A(1^2,\underline{2},2,1^3)$.
\end{proof}

By Theorem \ref{thm:crit_multi_grad_pairs}, we obtain the following proposition on tri-graded nilpotent tuples.
 \begin{corollary}
Let $\Lambda=(\lambda,\mu)$ be a flat tri-graded Young diagram. Then the number of $\Lambda$-graded nilpotent tuples (up to Levi-base change) is finite if and only if $\Lambda$ is listed in Lemma \ref{lem:3flat_rep_type} (a).
\end{corollary}
By Lemma \ref{lem:grad_fin_dim}, we also obtain finiteness criteria for tri-graded nilpotent tuples of a fixed tri-graded flat vector space.
\begin{lemma}\label{lem:3_fixed_vsp_nullroot}
Let $V$ be a flat tri-graded staircase algebra. If $\dimv V$ contains (up to correlation) one of the minimal nullroots of Figure \ref{fig:tri_min_nullroots}, then there are (up to Levi-base change) infinitely many tri-graded nilpotent tuples of $V$.
\end{lemma}

 For certain vector spaces, we obtain a concrete finiteness criterion on tri-graded nilpotent tuples.
\begin{lemma}\label{lem:3_fixed_vsp_tc}
Assume that $V$ is flat tri-graded of shape $\Lambda$ and $\A(\Lambda)$ is tame concealed. Then there are only finitely many tri-graded nilpotent tuples on $V$ (modulo base change in the homogeneous components) if and only if $\underline{\dim}V$ does not contain the corresponding minimal nullroot.
\end{lemma}
\begin{proof}
Any tame concealed algebra has a unique one-parameter family of indecomposable modules $X$ with $\End(X) = K$ and $\Ext^1(X,X) =K$; and the dimension vector of these modules is the minimal positive nullroot (known by \cite{HaVo}) which generates the radical of the quadratic form. 
\end{proof}
\begin{figure}
\begin{center}
\includegraphics[trim=130 495 120 125, clip,width=300pt]{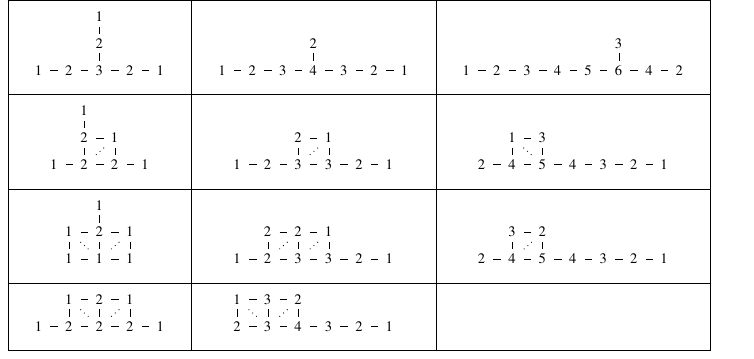}
\end{center}
\caption{Minimal nullroots of flat tri-graded nullroots}\label{fig:tri_min_nullroots}
\end{figure}
\section{Higher graded nilpotent tuples}\label{sect:high_grad}
In the following, we classify the representation types of higher graded staircase algebras, that is, $k$-graded staircase algebras, where $k\geq 4$. This completes the proof of Classification Theorem \ref{thm:class}.

We define  the algebra $\A(4):=\A(\Lambda)$, where
\[\begin{array}{ccc}
\Lambda&=&\{(1,1,1,1), (2,1,1,1),(1,2,1,1),\\&&(1,1,2,1),(1,1,1,2),(2,2,1,1),\\
&&(2,1,2,1),(1,2,1,2),(1,1,2,2) \}.
\end{array}\]

In terms of the quiver $\Q(4):=\Q(\Lambda)$, this means  $\A(4):=K\Q(4)/I$ and
\[\begin{tikzpicture}
\matrix (m) [matrix of math nodes, row sep=0.81em,
column sep=0.8em, text height=1.69ex, text depth=0.1ex]
{&& & (2,2,1,1)&  &\\
&&(2,1,1,1) & & (1,2,1,1) &\\
\Q(4) = &(2,1,2,1) &&(1,1,1,1)&&(1,2,1,2)\\
&& (1,1,2,1) & & (1,1,1,2) &\\
&& & (1,1,2,2)&  &\\
&& & f&  &\\
&&b & & c &\\
~~~~~ =: &i && a&&g\\
&& e & & d &\\
&& & h&  &\\
};
\path[->]
(m-1-4) edge (m-2-3)
(m-1-4) edge (m-2-5)
(m-5-4) edge (m-4-3)
(m-5-4) edge (m-4-5)
(m-3-2) edge (m-2-3)
(m-2-3) edge (m-3-4)
(m-2-5) edge (m-3-4)
(m-3-6) edge (m-2-5)
(m-3-2) edge (m-4-3)
(m-3-6) edge  (m-4-5)
(m-4-3) edge  (m-3-4)
(m-4-5) edge (m-3-4);
\path[-,dotted]
(m-3-2) edge (m-3-4)
(m-3-6) edge (m-3-4)
(m-1-4) edge (m-3-4)
(m-5-4) edge (m-3-4);
\path[->]
(m-6-4) edge (m-7-3)
(m-6-4) edge (m-7-5)
(m-10-4) edge (m-9-3)
(m-10-4) edge (m-9-5)
(m-8-2) edge (m-7-3)
(m-7-3) edge (m-8-4)
(m-7-5) edge (m-8-4)
(m-8-6) edge (m-7-5)
(m-8-2) edge (m-9-3)
(m-8-6) edge  (m-9-5)
(m-9-3) edge  (m-8-4)
(m-9-5) edge (m-8-4);
\path[-,dotted]
(m-8-2) edge (m-8-4)
(m-8-6) edge (m-8-4)
(m-6-4) edge (m-8-4)
(m-10-4) edge (m-8-4);
\end{tikzpicture}\]
The admissible ideal $I$ is generated by all commutative squares depicted in the picture.

\begin{theorem}
\label{thm:high_stair_rep_type}
A proper $k$-graded staircase algebra $\A(\Lambda)$, where $k\geq 4$, is 
\begin{itemize}
\item[(a)] never representation-finite,
\item[(b)] tame concealed if and only if $\Q(\Lambda)$ is of Euclidean type $\widetilde{D_4}$, 
\item[(c)] tame, but not tame concealed if and only if  its Gabriel quiver is a subquiver of $\Q(4)$ with the induced relations, but not of Euclidean type $\widetilde{D_4}$.
\end{itemize} 
 Otherwise, and in particular if $k\geq 5$, the algebra $\A(\Lambda)$ is of wild representation type.
\end{theorem}
\begin{proof}
The algebra $\A(\Lambda)$  contains a subcategory of Euclidean type $\widetilde{D}_4$, namely
\[\begin{tikzpicture}
\matrix (m) [matrix of math nodes, row sep=0.2em,
column sep=0.6em, text height=0.6ex, text depth=0.1ex]
{\bullet &&\bullet \\
&  \bullet &\\
\bullet &  & \bullet\\ };
\path[->]
(m-1-1) edge (m-2-2)
(m-1-3) edge  (m-2-2)
(m-3-1) edge (m-2-2)
(m-3-3) edge (m-2-2);
\end{tikzpicture}\]
which corresponds to a tame concealed algebra. Thus, $\A(\Lambda)$ is always of infinite representation type and not tame concealed by \cite{HaVo} (see Subsection \ref{ssect:theory}) except in this one special case. We have to show that $\A(4)$ is of tame representation type. This fact follows from Theorem \ref{thm:crit_tame}, since the algebra does not contain any hypercritical subcategory.  
Since $\A(4)$ is strongly simply connected, by Lemma \ref{thm:crit_tame}, we can also show tameness straight away, since its quadratic form is non-negative:
\[\begin{array}{ll}
q(a,b,c,d,e,f,g,h,i)&= (\frac{a}{2}-b+\frac{f}{2}+ \frac{i}{2})^2+ (\frac{a}{2}-e+\frac{h}{2}+ \frac{i}{2})^2\\
& ~~~~ + (\frac{a}{2}-c+\frac{f}{2}+ \frac{g}{2})^2+ (\frac{a}{2}-d+\frac{h}{2}+ \frac{g}{2})^2\\
& ~~~~ + (\frac{f}{2}-\frac{i}{2})^2+ (\frac{f}{2}-\frac{g}{2})^2+ (\frac{h}{2}-\frac{i}{2})^2+ (\frac{h}{2}-\frac{g}{2})^2\\
& \geq 0.
\end{array} \]
In every remaining case, the Gabriel quiver of $\A(\Lambda)$ contains a hypercritical subquiver of the form 
\[\begin{tikzpicture}
\matrix (m) [matrix of math nodes, row sep=0.2em,
column sep=0.6em, text height=0.6ex, text depth=0.1ex]
{\bullet &&\bullet \\
&  \bullet &\\
\bullet &  & \bullet & \bullet\\ };
\path[->]
(m-1-1) edge (m-2-2)
(m-1-3) edge  (m-2-2)
(m-3-1) edge (m-2-2)
(m-3-3) edge (m-2-2)
(m-3-4) edge (m-3-3);
\end{tikzpicture}\]
or a wild subquiver
\[\begin{tikzpicture}
\matrix (m) [matrix of math nodes, row sep=1.81em,
column sep=1.8em, text height=0.9ex, text depth=0.1ex]
{\bullet &\bullet &\bullet & \bullet & \bullet\\
& &  \bullet &&\\
 };
\path[->]
(m-1-1) edge (m-2-3)
(m-1-2) edge  (m-2-3)
(m-1-3) edge  (m-2-3)
(m-1-4) edge (m-2-3)
(m-1-5) edge (m-2-3);
\end{tikzpicture}\]
Then $\A(\Lambda)$ is of wild representation type by \cite{Un} and Lemma \ref{lem:reduct_subcat}.
\end{proof}

For $k$-graded vector spaces, where $k\geq 4$, we immediately obtain a finiteness criterion.

\begin{corollary}\label{cor:high_grad_inf_tuples}
Let $\Lambda$ be a generalized $k$-dimensional partition, where $k\geq 4$. Then there are (up to Levi-base change) infinitely many $\Lambda$-graded nilpotent tuples.
\end{corollary}

\begin{lemma}
Let $V$ be a $4$-graded staircase algebra. If $\dimv V$ contains (up to correlation) the minimal nullroot
\begin{center}
\begin{tikzpicture}
\matrix (m) [matrix of math nodes, row sep=0.2em,
column sep=0.6em, text height=0.6ex, text depth=0.1ex]
{_1 &&_1 \\
&  _2 &\\
_1 &  & _1\\ };
\path[->]
(m-1-1) edge (m-2-2)
(m-1-3) edge  (m-2-2)
(m-3-1) edge (m-2-2)
(m-3-3) edge (m-2-2);
\end{tikzpicture}
\end{center}

then there are (up to Levi-base change) infinitely many $4$-graded nilpotent tuples of $V$. If the shape of $V$ equals the shape of the nullroot, the converse imlication is true, as well, since the algebra is tame concealed.
\end{lemma}
\begin{proof}
The first part follows from Lemma \ref{lem:grad_fin_dim}; the proof of the second part is similar to the proof of Lemma \ref{lem:3_fixed_vsp_tc}.
\end{proof}

\appendix

\section{Flat tri-graded algebras}\label{app:arqs}

\subsection[Auslander-Reiten quiver of A(2,2,1,1,1,1)]{Auslander-Reiten quiver of $\A(2,\underline{2},1^4)$}\label{app:arq_221111}

\begin{center}
\includegraphics[trim=35 480 260 160, clip,width=300pt]{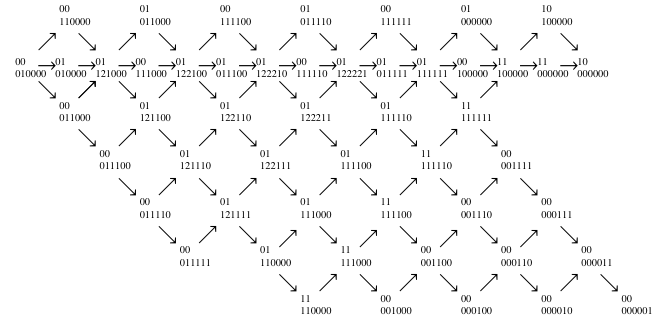}
\end{center}

\subsection[Auslander-Reiten quiver of A(2,2,2,2,1)]{Auslander-Reiten quiver of $\A(2,2,2,\underline{2},1)$}\label{app:arq_12222}
\begin{center}
\includegraphics[angle=270, trim=150 5 320 460, clip,width=300pt]{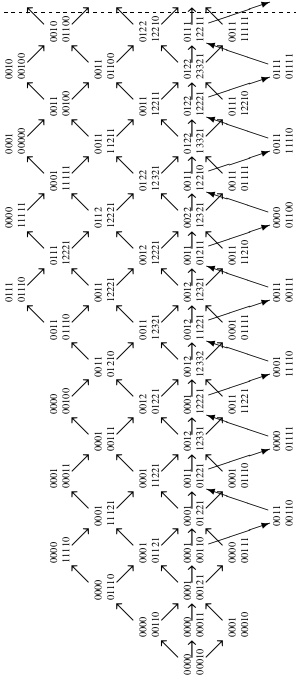}
\end{center}

\begin{center}
\includegraphics[angle=270, trim=130 310 320 160, clip,width=300pt]{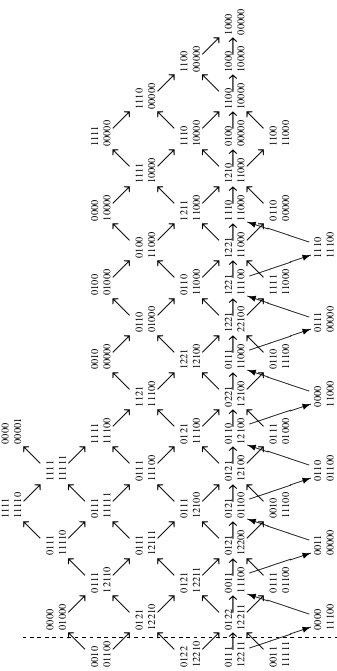}
\end{center}

\subsection[Auslander-Reiten quiver of A(3,3,1,1,1)]{Auslander-Reiten quiver of $\A(3,\underline{3},1,1,1)$}\label{app:arq_33111}
\begin{center}
\includegraphics[angle=270, trim=140 30 310 410, clip,width=300pt]{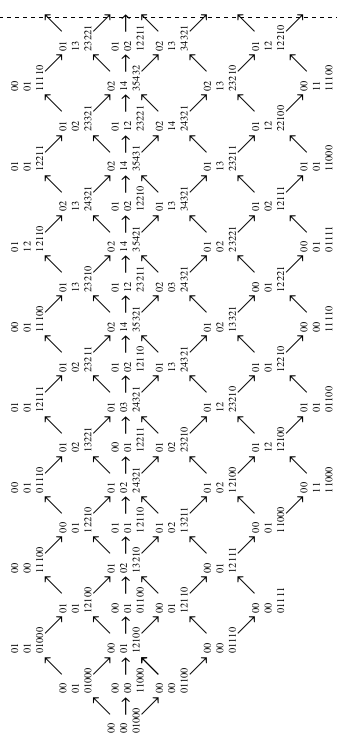}
\end{center}

\begin{center}
\includegraphics[angle=270, trim=140 350 280 90, clip,width=300pt]{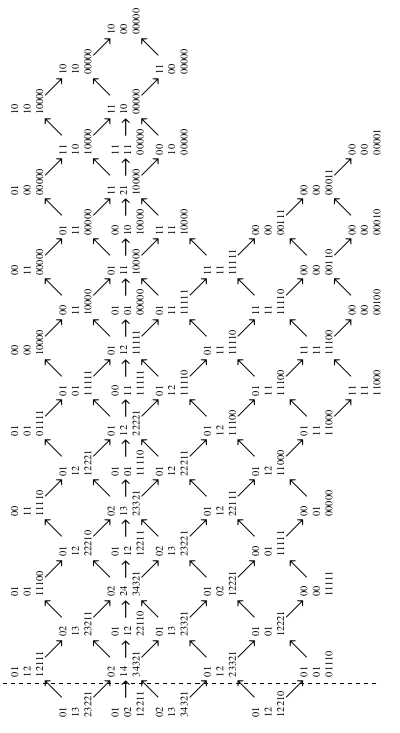}
\end{center}

\subsection{Case study (up to symmetry)}\label{app:case_study}
\begin{center}
\includegraphics[trim=130 560 124 125, clip,width=300pt]{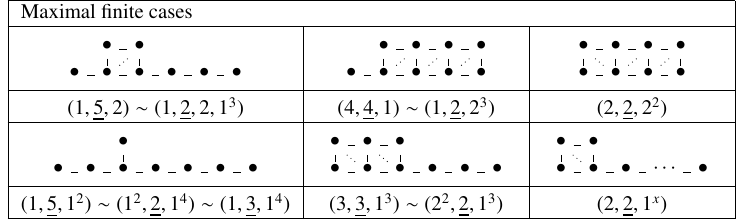}
\end{center}

\begin{center}
\includegraphics[trim=130 381 115 125, clip,width=300pt]{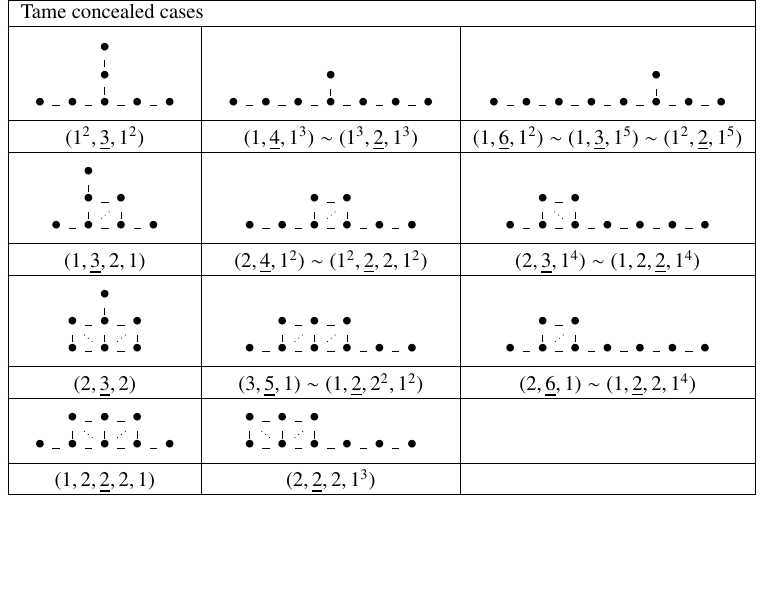}
\end{center}

\begin{center}

\includegraphics[trim=130 490 145 125, clip,width=300pt]{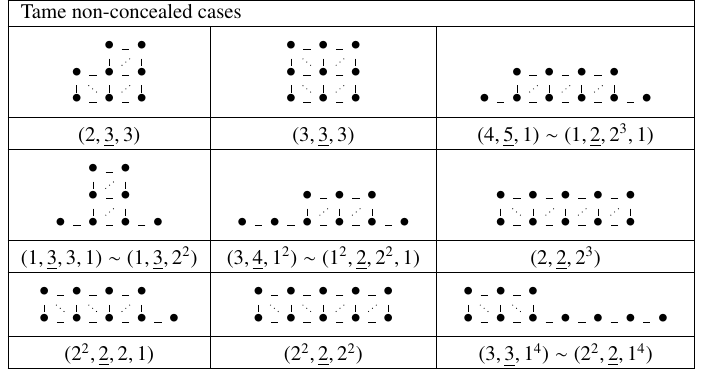}
\end{center}

\end{document}